\documentclass[12pt]{article}
\nonstopmode
\RequirePackage[colorlinks,citecolor=blue,urlcolor=blue,linkcolor=blue]{hyperref}
\hypersetup{
colorlinks = true,
citecolor=blue,
urlcolor=blue,
linkcolor=blue,
pdfauthor = { Alexey Kuznetsov},
pdfkeywords = { discrete Fourier transform, orthogonal basis, eigenvectors, Hermite functions, q-binomial theorem},
pdftitle = {Explicit Hermite-type eigenvectors of the discrete Fourier transform},
pdfpagemode = UseNone
}
\usepackage{graphicx,xspace,colortbl}
\usepackage{amsmath,amsthm,amsfonts}
\usepackage[titletoc,title]{appendix}
\usepackage{color}
\usepackage{fancybox}
\usepackage{epsfig}
\usepackage{subfig}
\usepackage{pdfsync}
    \def\qed{\hfill$\sqcap\kern-8.0pt\hbox{$\sqcup$}$\\}
    \def\beq{\begin{eqnarray}}
    \def\eeq{\end{eqnarray}}
    \def\beqq{\begin{eqnarray*}}
    \def\eeqq{\end{eqnarray*}}

    \def\r{{\mathbb R}}
    \def\c{{\mathbb C}}
    	
    \def\d{{\textnormal d}}
    \def\i{{\textnormal i}}

\newtheorem{theorem}{Theorem}
\newtheorem{lemma}{Lemma}

\newtheorem{corollary}{Corollary}
\theoremstyle{definition}
\newtheorem{definition}{Definition}

\newtheorem{remark}{Remark}

\title{Explicit Hermite-type eigenvectors of the discrete Fourier transform}
\author{
{Alexey Kuznetsov
\footnote{Dept. of Mathematics and Statistics,  York University,
4700 Keele Street, Toronto, ON, M3J 1P3, Canada.  \newline
E-mail:  kuznetsov@mathstat.yorku.ca  \newline Research supported by the Natural Sciences and Engineering Research Council of Canada. } 
 }}
 
 \date{\today}

\begin{document}
\maketitle

\begin{abstract} 
The search for a canonical set of eigenvectors of the discrete Fourier transform has been ongoing for more than three decades. 
The goal is to find an orthogonal basis of eigenvectors which would approximate Hermite functions -- the eigenfunctions of the continuous Fourier transform. This eigenbasis should also have some degree of analytical tractability and should allow for 
efficient numerical computations. In this paper we provide a partial solution to these problems. 
First, we construct an explicit basis of (non-orthogonal) eigenvectors of the discrete Fourier transform, thus extending the results of \cite{Kong_2008}. Applying the Gramm-Schmidt orthogonalization procedure we obtain an orthogonal eigenbasis of the discrete Fourier transform. We prove that the first eight eigenvectors converge to the corresponding Hermite functions, and we conjecture that this convergence result remains true for all eigenvectors. 
\end{abstract}

{\vskip 0.15cm}

\section{Introduction}

The fractional Fourier transform is becoming increasingly more important due to an ever-growing list of applications in signal processing, optics and quantum mechanics and in other areas of Science \cite{bultheel2007,FRFT_book}. In order to define this object one needs first to diagonalize the Fourier transform operator ${\mathcal F}$, and that is where 
Hermite functions 
\begin{equation}\label{def_Hermite_psin}
\psi_n(x):=(-1)^n (\sqrt{\pi} 2^n n!)^{-1/2}  e^{x^2/2} \frac{\d^n}{\d x^n} e^{-x^2}, \;\;\; n\ge 0  
\end{equation}
become indispensable. It is well-known that Hermite functions are the eigenfunctions of the Fourier transform operator
$$
({\mathcal F} \psi_n)(x):=\frac{1}{\sqrt{2\pi}}\int_{\r} e^{-\i x y} \psi_n(y) \d y=(-\i)^n \psi_n(x),
$$
and that they form a complete orthonormal basis of $L_2(\r)$. In other words, the orthogonal basis of Hermite functions diagonalizes the Fourier transform operator ${\mathcal F}$ and this  allows us to define the fractional power $({\mathcal F})^a$ --
the fractional Fourier transform. Of course, there are infinitely many ways to choose an eigenbasis of the Fourier transform, each of them would lead to a different version of the fractional Fourier transform.  
What makes Hermite functions so special among all other eigenfunctions is that they are analytically tractable and they have many useful properties. This analytical tractability and Mehler's formula \cite{bultheel2007} (which gives an explicit expression for the integral kernel of the fractional Fourier transform) have sealed the standing of Hermite functions as the canonical set of the eigenfunctions for Fourier transform.

 \begin{table}
	\centering
	\begin{tabular}{| l || c | c | c | c |}
	\hline 	
	 	& dim$(W)$  & dim$(X)$  & dim$(Y)$ & dim$(Z)$    \\
	\hline \hline
	$N=4L$ &  $L+1$	  & $L$ & $L$ & $L-1$ \\ \hline
	$N=4L+1$ & 	 $L+1$ 	& $L$ & $L$ & $L$ \\ \hline
	$N=4L+2$ & 	 $L+1$ 	& $L$ & $L+1$ & $L$\\ \hline
	$N=4L+3$ &	 $L+1$	& $L+1$ & $L+1$ & $L$ \\ \hline
	\end{tabular}
	\caption{Dimensions of the eigenspaces $\{W, X, Y, Z\}$ of the discrete Fourier transform ${\mathcal F}_N$, corresponding to the eigenvalues
	$\{1, -\i, -1, \i\}$.}
	\label{tab_multiplicities}
\end{table}

Compared to the continuous Fourier transform, the situation with the eigenvectors of the discrete Fourier transform (DFT) is much more complex. Let us first introduce several necessary definitions and notations. 
The (centered) discrete Fourier transform ${\mathcal F}_N$ is a linear map that sends a vector 
${\mathbf u} \in \c^N$ into a vector 
${\mathbf v}
 \in \c^N$ according to  the rule
\begin{equation}\label{def_Fourier_transform}
v(l)= \frac{1}{\sqrt{N}}\sum\limits_{k \in I_N} e^{-2\pi \i kl/N} u(k), \;\;\; l\in I_N,
\end{equation} 
where $I_N$ denotes the set of integers 
$$I_N:=\{-M+1, -M+2, \dots , -M+N\}, \;\;\; M=\lfloor (N+1)/2 \rfloor.$$ 
Everywhere in this paper we will follow the convention that all vectors will be denoted by bold lower case letters and the elements of a vector will be labeled by the set $I_N$. Before we begin discussing the eigenvectors of the DFT, let us say a few words about its eigenvalues. The eigenvalues of both continuous and discrete Fourier transform are given by $\pm 1$, $\pm \i$. In the continuous case each eigenspace is infinite dimensional, whereas the multiplicities of eigenvalues in the discrete case
are presented in the Table \ref{tab_multiplicities}. This result was obtained in 1921 by  Schur, however it can be easily derived 
from a much earlier result of Gauss on the law of quadratic reciprocity, which is essentially a statement about the trace of the matrix corresponding to the linear transformation ${\mathcal F}_N$. The formula for the multiplicities of eigenvalues was rediscovered by McClellan and Parks 
\cite{McClellan1972} in 1972, and since then there have appeared many other proofs (see \cite{Murty2001} for a very simple proof based on the Vandermonde determinant formula).

Returning to our discussion of eigenvectors of the DFT, it seems that the first example of an eigenbasis was constructed by McClellan and Parks \cite{McClellan1972} as a 
by-product of their proof of the results in Table \ref{tab_multiplicities}. 
Other constructions of eigenvectors of the DFT include a number-theoretic construction \cite{Morton} and 
a representation theory approach \cite{Gurevich2009,Wang2010}. Mehta \cite{Mehta1987} has constructed an eigenbasis based on theta functions. One important approach for finding orthogonal eigenvectors of the DFT is based on certain symmetric tridiagonal (or almost tridiagonal) matrices which commute with the DFT. The eigenvectors of these matrices give an orthogonal eigenbasis of the DFT. Such matrices were first discovered by Dickinson and Steiglitz \cite{Dickinson} and by Grunbaum \cite{Grunbaum} (see also \cite{Candan,Clary_Mugler_2003,Santhanam2008} for more recent developments). 

The above discussion clearly shows that there is an abundance of different sets of eigenvectors of the DFT. How can we choose ``the best" among them? First we need to define what we mean by ``the best". An ideal set of eigenvectors of the discrete Fourier transform ${\mathcal F}_N$ should satisfy the following three requirements: the eigenvectors should be orthogonal, they should converge to the corresponding Hermite functions as $N\to +\infty$ and they should have at least some degree of analytical tractability (so that we are able to compute them).   
Judging by these standards, most of the above candidates can be disqualified. For example, the eigenbasis of McClellan and Parks is explicit but not orthogonal and it does not approximate Hermite functions; the eigenbasis of Mehta  \cite{Mehta1987} does approximate Hermite functions but it is not orthogonal; the eigenbasis constructed in \cite{Wang2010} is orthogonal, but it requires $N$
 to be a certain prime number and it does not seem to approximate Hermite functions.   In fact, the only sets of eigenvectors that do satisfy the above requirements come from the commuting symmetric matrix approach of Dickinson and Steiglitz \cite{Dickinson} and Grunbaum \cite{Grunbaum}, though in both cases the eigenvectors are not given explicitly and to find them one needs to diagonalize a tridiagonal symmetric matrix.

Not much is known about {\it explicit} Hermite-type eigenvectors of the DFT. In fact, it seems that the only result in this direction 
comes from the recent paper by Kong  \cite{Kong_2008}, who gives two examples of such eigenvectors.  
More precisely, when $N=4L+1$ the vector ${\mathbf u} \in \r^N$ with 
\begin{equation}\label{Kong_u}
u(k)=\prod\limits_{j=L+1}^{2L} \left( \cos(2\pi k/N)-\cos(2\pi j/N) \right), \;\;\; k \in I_N, 
\end{equation}
is an eigenvector of ${\mathcal F}_N$, and the same is true when $N=4L$ for the vector ${\mathbf v} \in \r^N$ with 
\begin{equation}\label{Kong_v}
v(k)=\sin(2\pi k/N)\prod\limits_{j=L+1}^{2L-1} \left( \cos(2\pi k/N)-\cos(2\pi j/N) \right), \;\;\; k \in I_N.
\end{equation}
 Numerical evidence given in \cite{Kong_2008} suggests that these two vectors approximate the first two Hermite functions. What is also interesting about these two eigenvectors is that they are unique in a certain sense. It follows from the proof of the main result in \cite{Kong_2008} (though it is not explicitly stated there) that the only eigenvector of the DFT of the dimension
 $N=4L+1$ which has zero elements $u(k)=0$ for $L+1 \le |k| \le 2L$ is the vector ${\mathbf u}$ defined in \eqref{Kong_u}.  

In this paper we provide the first known example of an explicit orthogonal eigenbasis of the DFT for which the individual eigenvectors 
seem to approximate the corresponding Hermite functions. First of all, we construct a basis of ${\mathbb R}^N$ which consists of 
even and odd vectors similar to Kong's vectors ${\mathbf u}$ and ${\mathbf v}$ given above. These vectors behave nicely under the DFT and we prove that they are unique in a certain sense. The proof of these results uses the methods of Kong \cite{Kong_2008} combined with the $q$-Binomial Theorem. From this basis of ${\mathbb R}^N$ we construct an eigenbasis of the DFT following the same procedure as 
McClellan and Parks \cite{McClellan1972}. As a corollary of this construction, we obtain a new proof of the result on the multiplicities of the eigenvalues of the DFT presented in Table \ref{tab_multiplicities}. Finally, we apply the Gramm-Schmidt orthogonalization procedure to the basis
of each eigenspace of the DFT. This gives us an orthonormal basis of ${\mathbb R}^N$, consisting of the eigenvectors of the 
discrete Fourier transform ${\mathcal F}_N$. 
We prove that the first eight eigenvectors converge to the corresponding Hermite functions
 as $N\to +\infty$ and we conjecture that this convergence holds true for all eigenvectors.

The paper is organized as follows. In Section \ref{section_results} we present our results for the case 
$N\equiv 1(\textnormal{mod } 4)$. The proofs of these results are collected in Section \ref{section_proofs} and 
in the Appendix we give the corresponding formulas for all remaining cases $N \in \{0,2,3\} (\textnormal{mod } 4)$.


\section{Results}\label{section_results}


We begin by stating some definitions and notations that will be used throughout this paper.
We denote by $W,X,Y,Z \subset \r^N$ the eigenspaces of the DFT corresponding to eigenvalues $1,-\i,-1,\i$ 
(their dimensions are given in Table \ref{tab_multiplicities}).  
The dot product of two vectors ${\mathbf u}$ and ${\mathbf v}$ will be denoted by ${\mathbf u} \cdot {\mathbf v}$ and the Euclidean norm of a vector in ${\mathbb C}^N$ is defined as $\|u\|^2=u \cdot  \bar u$. We denote by $\lfloor x \rfloor$ and $\lceil x \rceil$  the floor and the 
ceiling function. From now on we will drop the subscript $N$ in the notation of the discrete Fourier transform 
and will write simply ${\mathcal F}$ instead of ${\mathcal F}_N$ -- there should be no confusion as the continuous Fourier transform will not be used anymore. 

We say that a vector ${\mathbf a} \in {\mathbb R}^N$ is {\it even} ({\it odd}) if $a(k)=a(j)$ (respectively $a(k)=-a(j)$) for all $k,j\in I_N$ 
such that $k\equiv -j \; (\textnormal{mod }N )$. An equivalent definition can be given as follows: when $N=2M+1$ the even (odd) vectors satisfy the condition $a(k)=a(-k)$ (respectively, $a(k)=-a(k)$) for $-M \le k \le M$. When $N=2M$ the even vectors satisfy
$a(k)=a(-k)$ for $-M+1 \le k \le M-1$ (note that there is no restriction on $u(M)$) and the the odd vectors satisfy
$a(k)=-a(-k)$ for $-M+1 \le k \le M-1$ and $a(M)=0$. 
The following two properties are well known and will be used extensively later: if ${\mathbf u} \in \r^N$ is an even vector 
and ${\mathbf v}\in \r^N$ is an odd vector, then
\begin{itemize}
\item[(i)] ${\mathcal F}{\mathbf u}$ ($\i {\mathcal F}{\mathbf v}$)  is a real even (respectively, real odd) vector;
\item[(i)] ${\mathcal F}^2 {\mathbf u}={\mathbf u}$  and ${\mathcal F}^2 {\mathbf v}=-{\mathbf v}$;
\end{itemize}

The following object will appear frequently in this paper: the sequence  $\{S(k)\}_{k\ge 0}$ is defined by 
\begin{equation}\label{def_sk}
S(k):=\prod\limits_{j=1}^k 2 \sin(\pi j/N) \;\;\; \textnormal{ for } \; k\ge 1, 
\end{equation}
and $S(0):=0$. 
This sequence satisfies 
\begin{align}
&\;\textnormal{(i)} \;\; S(k) =0 \;\; \textnormal{ for } \; k\ge N; \label{sk_properties1}\\ 
&\textnormal{(ii)}  \;\; S(k) S(N-k-1)=N \;\; \textnormal{ for } 0\le k \le N-1. \label{sk_properties2} 
\end{align}
Property \eqref{sk_properties1} follows at once from the definition \eqref{def_sk} while property \eqref{sk_properties2} 
can be derived by writing the polynomial $P(w)=(1-w^N)/(1-w)$ as a product of $N-1$ linear factors and then taking the limit as $w\to 1$.

We denote by  ${\textnormal{supp}}({\textbf b})$ {\it the support} of a vector ${\mathbf b}$, which is the set of all indices $k \in I_N$ such that $b(k)\neq0$. Following Kong \cite{Kong_2008} we define  {\it the signal length} (or simply, {\it the length}) 
of a vector ${\textbf b}$ as 
$$
l({\mathbf b}):=\max\{i-j \; : \; i,j \in {\textnormal{supp}}({\textbf b})\}. 
$$
We record the following important relationship between the length of a vector and its DFT: 
For any nonzero vectors ${\mathbf b} \in \c^N$   we have
\begin{equation}\label{Kong_Thm1}
l({\mathbf b})+l({\mathcal F}{\mathbf b})\ge N+1.
\end{equation}
 This result (which reminds one of the uncertainty principle) follows from Theorem 1 in \cite{Kong_2008}. We include the proof of this result in Section \ref{section_proofs} for convenience of the reader. 

The next theorem is our first main result. It gives the exact number of even/odd vectors that are ``extreme" in the sense that the quantity $l({\mathbf b})+l({\mathcal F}{\mathbf b})$ is the smallest possible among all even/odd vectors.

\newpage

\begin{theorem}\label{thm_support} ${}$ Assume that $N\ge 3$. 
\begin{itemize}
\item[(i)] If $N=2M+1$ there exist only $M+1$ nonzero even vectors satisfying
$l( {\mathbf u})+l({\mathcal F}{\mathbf u})=N+1$ and only $M$ nonzero odd vectors such that 
$l( {\mathbf v})+l({\mathcal F}{\mathbf v}) \le N+3$. 
\item[(ii)] If $N=2M$ there exist only $M-1$ nonzero even vectors satisfying $u(M)=({\mathcal F}{\mathbf u})(M)=0$ and
$l( {\mathbf u})+l({\mathcal F}{\mathbf u})\le N+2$ and only $M-1$ nonzero odd vectors such that
$l( {\mathbf v})+l({\mathcal F}{\mathbf v}) \le N+2$.  
\end{itemize}
\end{theorem}

The results and definitions that we have presented so far are valid for all integer values of $N \ge 1$. 
From now on we will concentrate on the case when
$N\equiv 1 \, ({\textnormal{mod }}4)$. The corresponding results for all remaining cases 
$N\in \{0,2,3\} \, ({\textnormal{mod }}4)$ are collected in the Appendix.

\subsection{The case when $N=4L+1$}

Our next goal is to find explicit expression for the vectors described in Theorem  \ref{thm_support} and to compute their discrete Fourier transform. 

\begin{theorem}\label{theorem_4N_plus1_n1}
${}$
\begin{itemize}
\item[(i)] The even vectors $\{{\mathbf u}_n\}_{-L \le n \le L}$ described in Theorem \ref{thm_support}(i) are given by 
\begin{equation}\label{def_un}
u_n(k):= S(3L+n+k)S(3L+n-k) \;\;\;\textnormal{ for } -L \le n \le L \; \textnormal{ and } \; k \in I_N. 
\end{equation}
These vectors are linearly independent and they satisfy
\begin{equation}\label{eqn_Fu_n}
{\mathcal F} {\mathbf u}_n=N^{-1/2}S(2L+2n) {\mathbf u}_{-n}.
\end{equation}
\item[(ii)] The odd vectors $\{{\mathbf v}_n\}_{-L \le n \le L-1}$ described in Theorem \ref{thm_support}(i)
are given by
\begin{equation}\label{def_vn}
v_n(k):=\sin(2\pi k/N) S(3L+n+k)S(3L+n-k)\;\;\;\textnormal{ for } -L \le n \le L-1 \; \textnormal{ and } \; k \in I_N. 
\end{equation}
These vectors are linearly independent and they satisfy
\begin{equation}\label{eqn_Fv_n}
{\mathcal F} {\mathbf v}_n=-\i N^{-1/2} S(2L+2n+1) {\mathbf v}_{-n-1}. 
\end{equation}
\end{itemize}
\end{theorem}

\begin{figure}
\centering
\subfloat[][]{\label{L10_f1}\includegraphics[height =6cm]{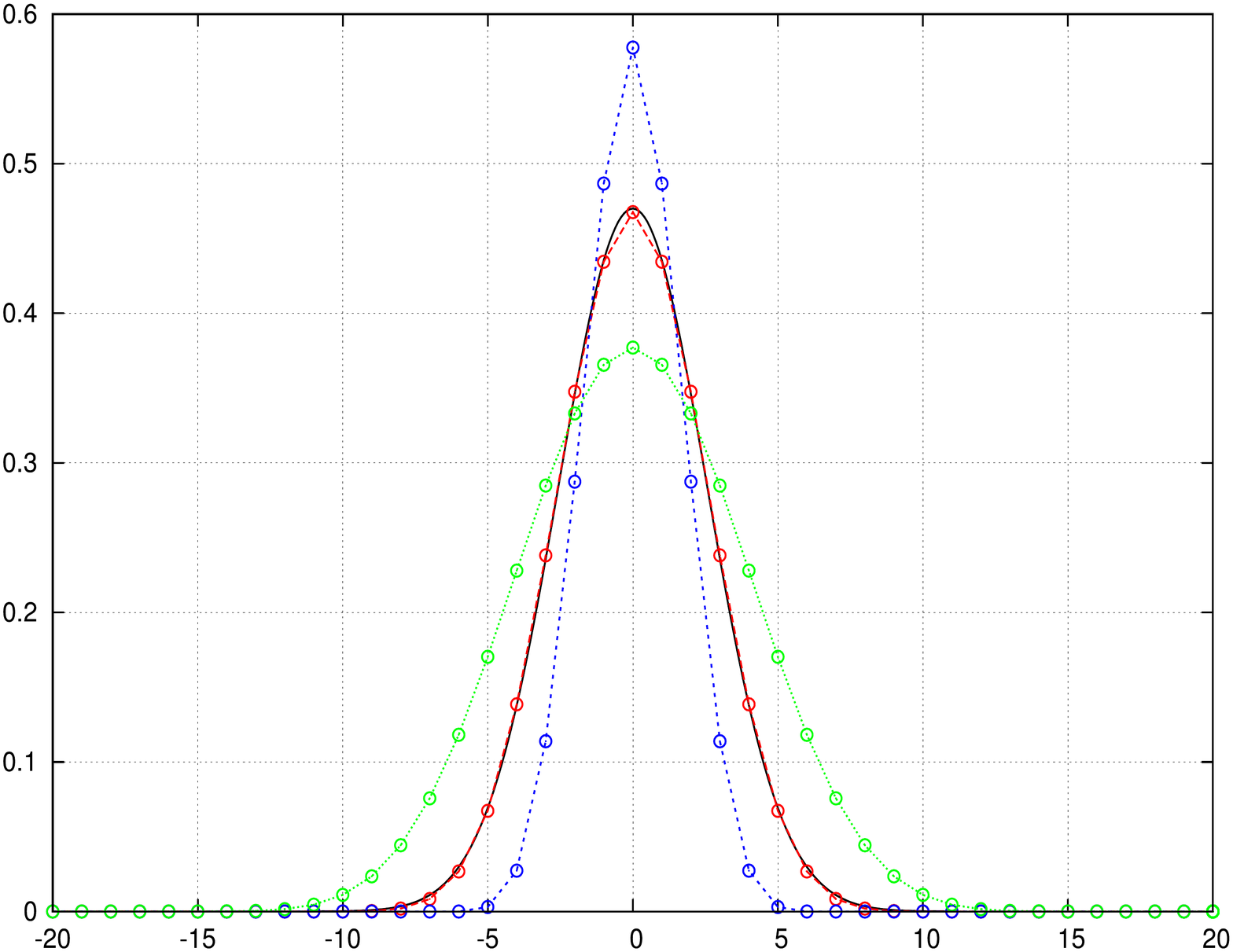}} \qquad
\subfloat[][]{\label{L10_f2}\includegraphics[height =6cm]{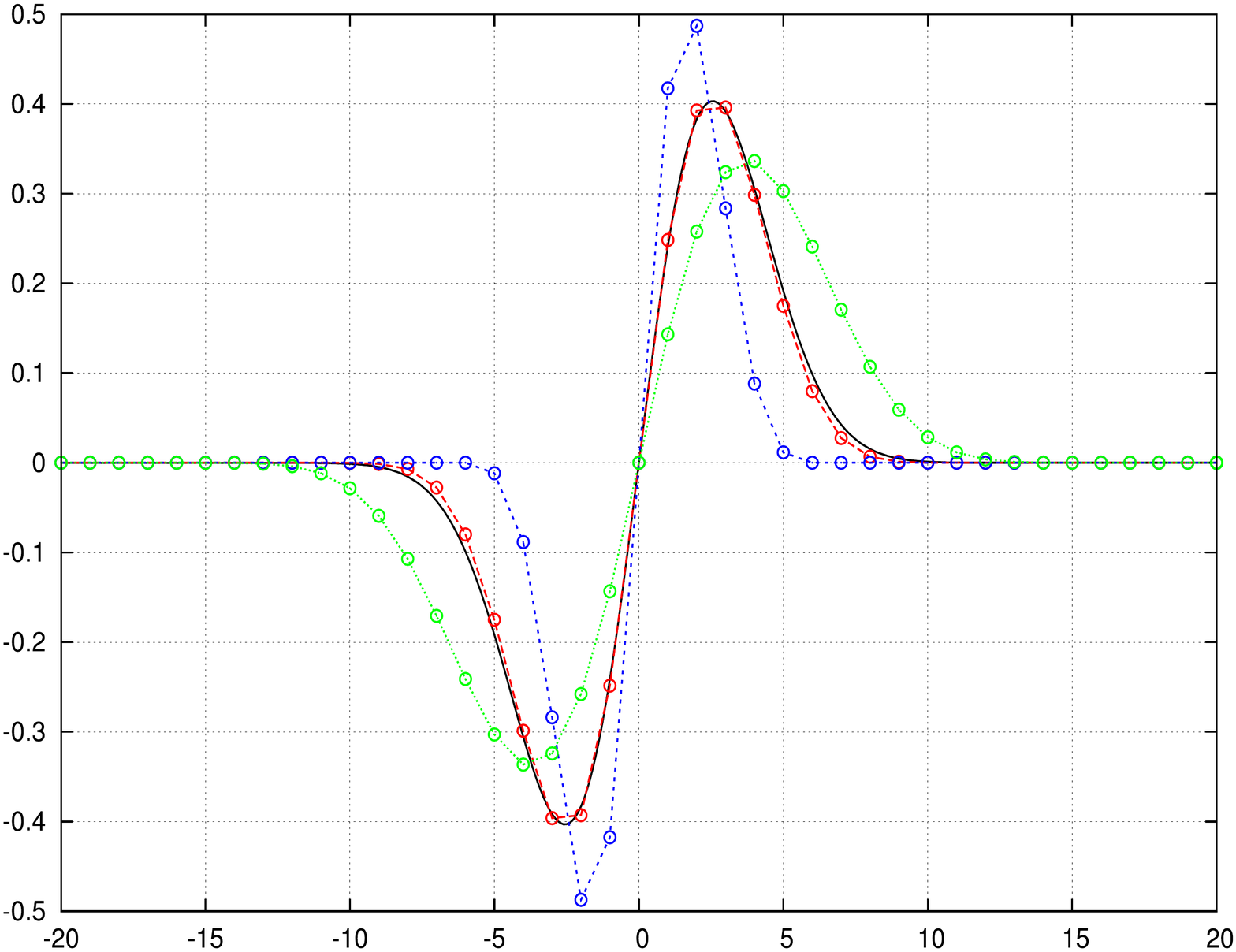}} 
\caption{
The red (respectively green, blue)  circles correspond to  (a) ${u}_n(k)/\|\mathbf u_n\|$ (b) ${v}_n(k)/\|\mathbf v_n\|$ for 
$n=0$ (respectively $n=-5$, $n=5$) and the black curve is the graph of the function (a) 
$x\mapsto (2/N)^{1/4} \exp(-\pi x^2/N)$ and (b)
$x\mapsto 2^{5/4}\pi^{1/2} N^{-3/4} x\exp(-\pi x^2/N)$. Here $N=41$ and the $x$-axis represents $k \in I_N$. 
}
\label{fig_L40}
\end{figure}

\begin{remark} Using the sum-to-product formula for the cosine function and identity \eqref{sk_properties2} it is easy to check that
formula \eqref{def_un} is equivalent to
\begin{equation}\label{u_n_product_formula}
u_n(k)=N2^{L+n} \prod\limits_{j=L-n+1}^{2L} \left( \cos(2\pi k/N)-\cos(2\pi j/N)\right). 
\end{equation}
In particular, the vector ${\mathbf u}_0$ is a scalar multiple of the vector ${\mathbf u}$ discovered by 
Kong \cite{Kong_2008} (see \eqref{Kong_u}). 
\end{remark}

Let us highlight some important properties of the vectors ${\mathbf u}_n$ and ${\mathbf v}_n$.
First of all, 
property \eqref{sk_properties1} implies that
\begin{equation}\label{support_uv}
{\textnormal{supp}}({\textbf u}_n)=\{k \in {\mathbb Z} \; : |k| \le L-n\}
\;\; {\textnormal{ and   }}\;\; {\textnormal{supp}}({\textbf v}_n)={\textnormal{supp}}({\textbf u}_n) \setminus\{0\}.
\end{equation}
In particular, we have $l({\textbf u_n})=l({\textbf v_n})=2L-2n+1$ and it is clear that 
the  vectors ${\mathbf u}_n$ and ${\mathbf v}_n$ satisfy 
$$
l({\mathbf u}_n)+l({\mathcal F}{\mathbf u}_n)=N+1, \;\;\; 
l({\mathbf v}_n)+l({\mathcal F}{\mathbf v}_n)=N+3,
$$ 
so these are precisely the vectors described in Theorem \ref{thm_support}.
The vectors ${\mathbf u}_n$ and ${\mathbf v}_n$ should be viewed 
as the discrete counterparts of the Gaussian functions $\exp(-a x^2)$
and $x\exp(-ax^2)$ (see Figure \ref{fig_L40}).  
There are indeed many analogies. First of all,  the elements of these vectors, $u_0(k)$ and $v_0(k)$, 
 converge to the corresponding Gausian functions as $N\to +\infty$ (this follows from Lemma \ref{lemma_v0} in Section
 \ref{section_proofs}). 
Second, the identities \eqref{eqn_Fu_n} and \eqref{eqn_Fv_n} are discrete analogues of the integral formulas
$$
 \int_{\r} e^{- a x^2} \times e^{-\i z x } \d x=\sqrt{\pi/a} \times e^{-z^2/(4a)}, 
 \;\;\;
 \int_{\r} x e^{- a x^2} \times e^{-\i z x}  \d x=-\i \sqrt{\pi/(4a^3)} \times z e^{-z^2/(4a)}. 
$$
Finally, the vectors ${\mathbf u}_n$ and ${\mathbf v}_n$ are connected by the following identity
\begin{equation}\label{vn_k1_minus_k2}
u_n(k+1)-u_n(k-1)=- 4\sin(\pi(2L+2n)/N)v_{n-1}(k), \;\;\; 
{\textnormal { for $-L<n\le L$ and $-2L < k < 2L$}},
\end{equation}
which we will establish in Section \ref{section_proofs}. This 
result is the finite difference  counterpart of the formula
$$
\frac{\d}{\d x} e^{-a x^2} = -2 a x e^{-a x^2}. 
$$

Our next goal is to construct a basis for each eigenspace $W$, $X$, $Y$ and $Z$ of the DFT.  
This construction is based on Theorem \ref{theorem_4N_plus1_n1} and the following property: if ${\mathbf u}$
 is an even vector  (${\mathbf v}$ is an odd vector) then 
${\mathcal F} {\mathbf u} \pm {\mathbf u}$ 
(respectively, $\i {\mathcal F} {\mathbf v} \pm {\mathbf v}$) is an eigenvector 
of the discrete Fourier transform ${\mathcal F}$ with the corresponding eigenvalue $\pm 1$ (respectively, $\mp \i$). This property
 can be easily established using 
the fact that ${\mathcal F}^2 {\mathbf u}={\mathbf u}$ and ${\mathcal F}^2 {\mathbf v}=-{\mathbf v}$ for every even vector ${\mathbf u}$ and 
every odd vector ${\mathbf v}$.

\begin{corollary}\label{prop_4N_plus1_n2}
Define $N$ vectors $\{{\mathbf w}_n\}_{0\le n \le L}$, $\{{\mathbf x}_n\}_{0\le n \le L-1}$, 
$\{{\mathbf y}_n\}_{0\le n \le L-1}$, $\{{\mathbf z}_n\}_{0\le n \le L-1}$ as follows:
\begin{align}\label{def_wxyz}
\nonumber
&{\mathbf w}_n={\mathbf u}_n+N^{-1/2} S(2L+2n) {\mathbf u}_{-n}\;, \;\; 0 \le n \le L,\\
\nonumber
&{\mathbf x}_n={\mathbf v}_n+N^{-1/2} S(2L+2n+1) {\mathbf v}_{-n-1}\;, \;\; 0\le n \le L-1,\\
&{\mathbf y}_n=-{\mathbf u}_{n+1}+N^{-1/2} S(2L+2n+2) {\mathbf u}_{-n-1}\; , \;\; 0 \le n \le L-1,\\
\nonumber
&{\mathbf z}_n=-{\mathbf v}_n+N^{-1/2} S(2L+2n+1) {\mathbf v}_{-n-1}\;, \;\; 0\le n \le L-1.
\end{align}
The vectors $\{{\mathbf w}_n\}_{0\le n \le L}$ (respectively, $\{{\mathbf x}_n\}_{0\le n \le L-1}$, 
$\{{\mathbf y}_n\}_{0\le n \le L-1}$ and $\{{\mathbf z}_n\}_{0\le n \le L-1}$)
form a basis for the eigenspace $W$ (respectively, $X$, $Y$ and $Z$). 
\end{corollary}

\begin{remark}
The results of Corollary \ref{prop_4N_plus1_n2} imply that the eigenvalues $\{1,-\i ,-1,\i\}$ of the discrete Fourier transform have corresponding multiplicities $\{L+1,L,L,L\}$, which agrees with Table \ref{tab_multiplicities}. 
\end{remark}

The eigenbasis  constructed in Proposition \ref{prop_4N_plus1_n2} is not orthogonal. Our next step is to 
build an orthonormal basis via the Gramm-Schmidt algorithm.

\begin{figure}[t]
\centering
\subfloat[][$n=0$ (black) and $n=1$ (red)]{\label{L10_f1}\includegraphics[height =4.7cm]{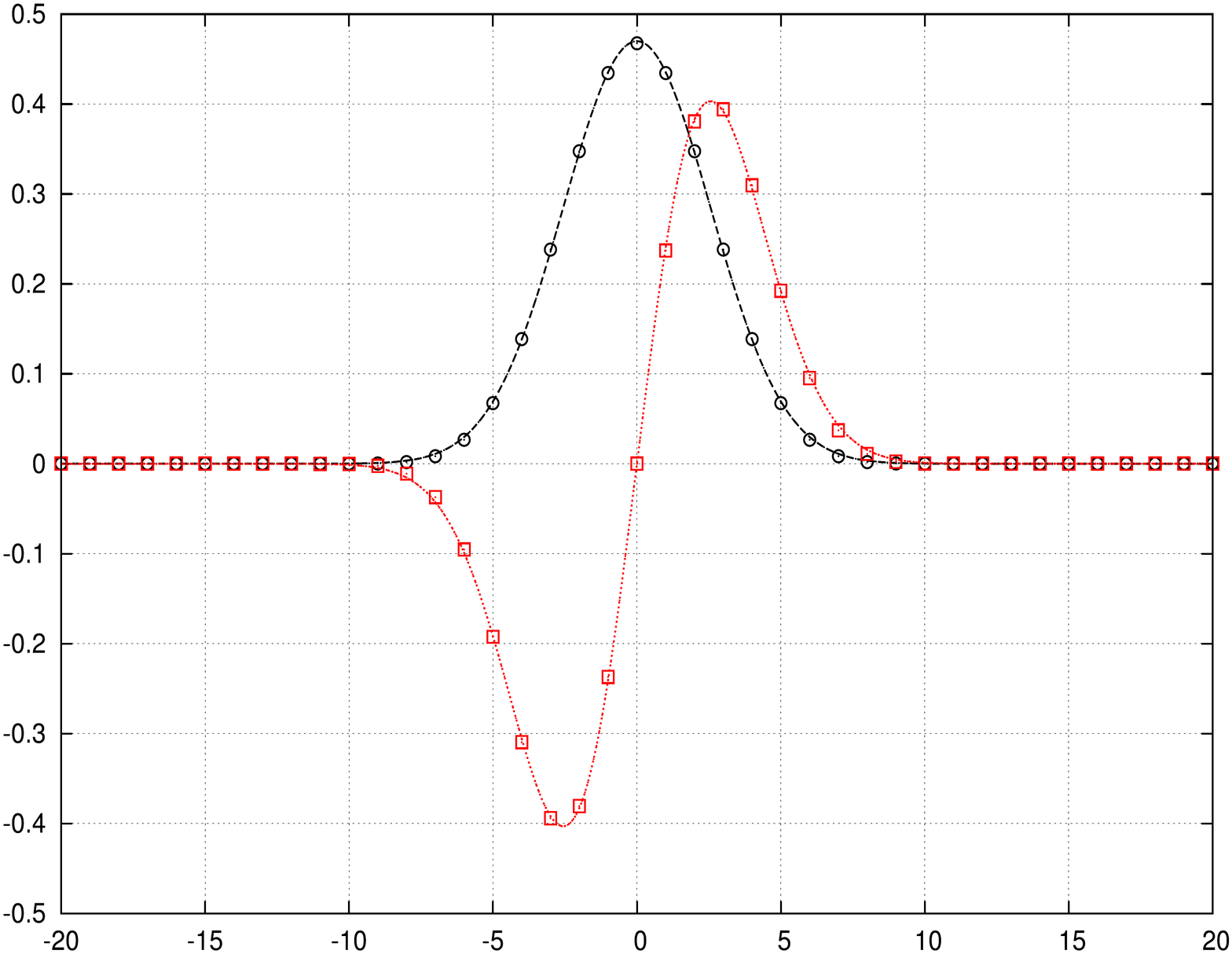}} 
\subfloat[][$n=2$ (black) and $n=3$ (red)]{\label{L10_f2}\includegraphics[height =4.7cm]{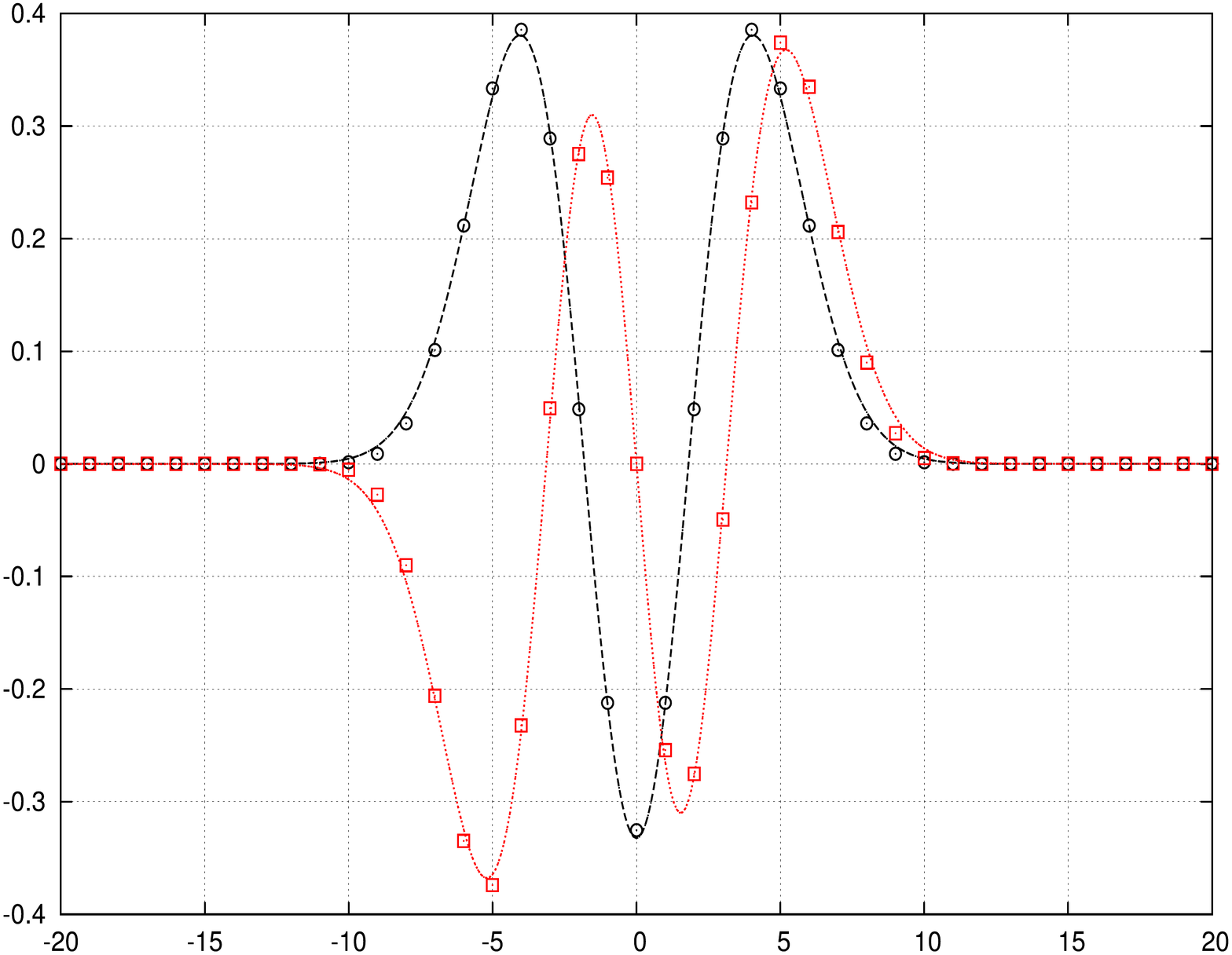}} 
\subfloat[][$n=4$ (black) and $n=5$ (red)]{\label{L10_f2}\includegraphics[height =4.7cm]{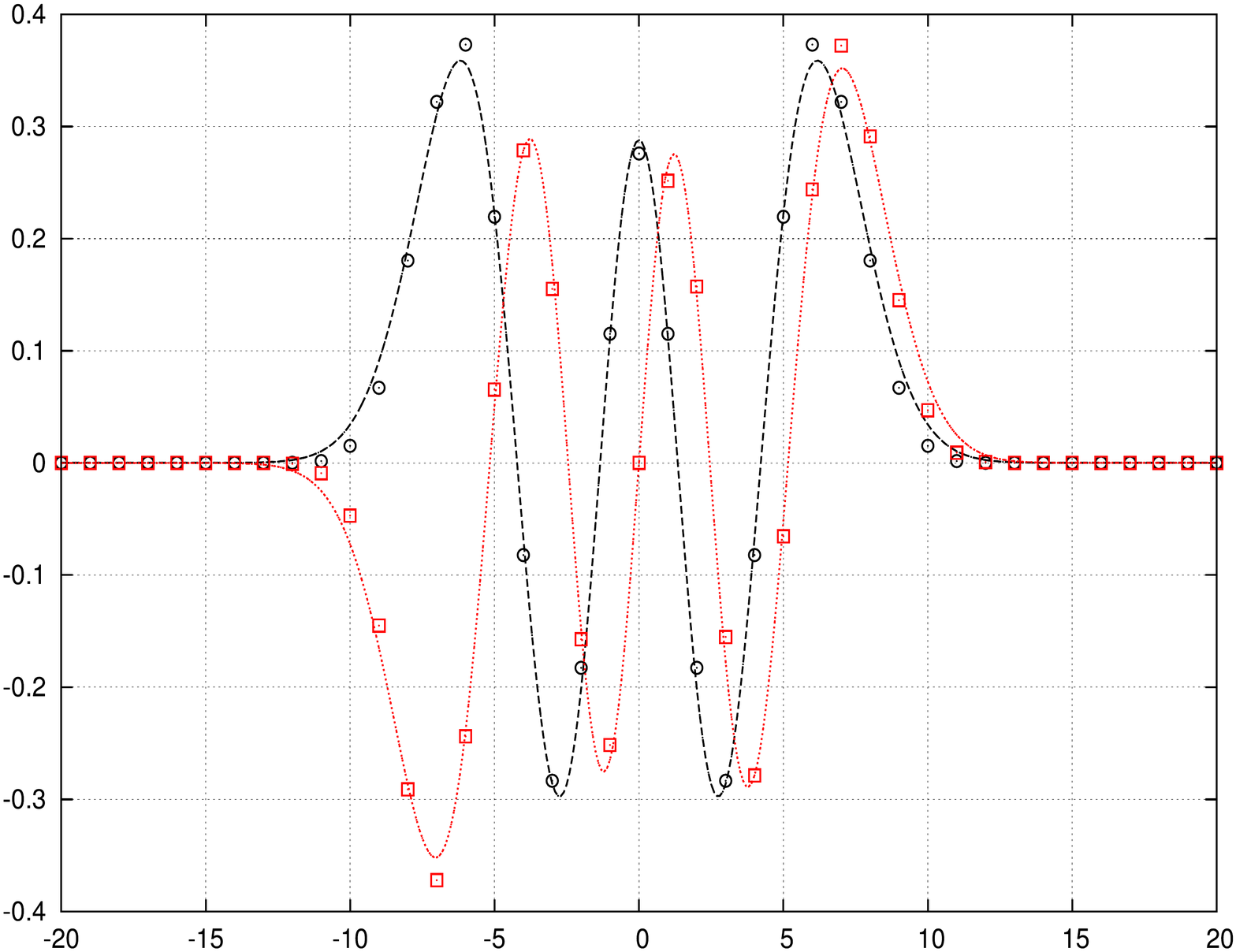}} 
\caption{Comparing the first six eigenvectors $\phi_n(k)$ (black circles and red squares) with the scaled Hermite functions $\epsilon^{1/2}\psi_n(\epsilon k)$ (black and red  lines). Here $N=41$, $\epsilon=\sqrt{2\pi /N}$ and the $x$-axis represents 
the index $k \in [-20,20]$.}
\label{fig_L10}
\end{figure}

\begin{definition}\label{def_phi}
Define $N$ vectors $\{\boldsymbol\phi_k\}_{0\le k \le N-1}$ as follows:
\begin{equation}
\boldsymbol\phi_{4n}=\frac{\tilde{\mathbf w}_n}{\|\tilde{\mathbf w}_n\|}, \;\;\; 
\boldsymbol\phi_{4n+1}=\frac{\tilde{\mathbf x}_n}{\|\tilde{\mathbf x}_n\|}, 
\;\;\; \boldsymbol\phi_{4n+2}=\frac{\tilde{\mathbf y}_n}{\|\tilde{\mathbf y}_n\|}, 
\;\;\; \boldsymbol\phi_{4n+3}=\frac{\tilde{\mathbf z}_n}{\|\tilde{\mathbf z}_n\|}, 
\end{equation}
where the vectors $\{\tilde {\mathbf w}_n\}_{0\le n \le L}$, $\{\tilde{\mathbf x}_n\}_{0\le n \le L-1}$, 
$\{\tilde{\mathbf y}_n\}_{0\le n \le L-1}$, $\{\tilde{\mathbf z}_n\}_{0\le n \le L-1}$ 
are obtained by Gramm-Schmidt algorithm starting from
$\{{\mathbf w}_n\}_{0\le n \le L}$, $\{{\mathbf x}_n\}_{0\le n \le L-1}$, 
$\{{\mathbf y}_n\}_{0\le n \le L-1}$, $\{{\mathbf z}_n\}_{0\le n \le L-1}$: 
\begin{align*}
&\tilde {\mathbf w}_m={\mathbf w}_m-\sum\limits_{j=0}^{m-1} 
\frac{{\mathbf w}_m \cdot \tilde {\mathbf w}_j}{\|\tilde {\mathbf w}_j\|^2}
\tilde {\mathbf w}_j, \;\; 0 \le n \le L, \\
&\tilde {\mathbf x}_n={\mathbf x}_n-\sum\limits_{j=0}^{n-1} 
\frac{{\mathbf x}_n\cdot\tilde {\mathbf x}_j}{\|\tilde {\mathbf x}_j\|^2}
\tilde {\mathbf x}_j,  \;\; 0 \le n \le L-1, 
\end{align*}
\begin{align*}
&\tilde {\mathbf y}_n={\mathbf y}_n-\sum\limits_{j=0}^{n-1} 
\frac{{\mathbf y}_n\cdot\tilde {\mathbf y}_j}{\|\tilde {\mathbf y}_j\|^2}
\tilde {\mathbf y}_j, \;\; 0 \le n \le L-1, \\
&\tilde {\mathbf z}_n={\mathbf z}_n-\sum\limits_{j=0}^{n-1} 
\frac{{\mathbf z}_n\cdot\tilde {\mathbf z}_j}{\|\tilde {\mathbf z}_j\|^2}
\tilde {\mathbf z}_j, \;\; 0 \le n \le L-1. \\
\end{align*}
\end{definition}

The following theorem is our main result in this paper. Recall that Hermite functions $\psi_n(x)$ are defined by \eqref{def_Hermite_psin}. 
\begin{theorem}\label{thm_convergence}
${}$
\begin{itemize}
\item[(i)]
The vectors $\{\boldsymbol\phi_n\}_{0\le n \le N-1}$ form an orthonormal basis in $\r^N$ consisting of eigenvectors of the discrete Fourier transform: for all $0\le n \le N-1$ we have ${\mathcal F} {\boldsymbol\phi}_n =  (-\i)^n {\boldsymbol\phi}_n$
\item[(ii)] 
For $0\le n \le 7$ we have
\begin{equation}\label{convergence_phi_psi}
\max_{-2L \le k \le 2L}  |\epsilon^{-1/2} \phi_n(k)-\psi_n(\epsilon k)| \to 0 \;
{\textnormal{ as }} N\to +\infty,
\end{equation}
where $\epsilon:=\sqrt{2\pi /N}$.
\end{itemize}
\end{theorem}

\begin{figure}[t]
\centering
\subfloat[][$n=0$ (black) and $n=1$ (red)]{\label{L100_f1}\includegraphics[height =4.7cm]{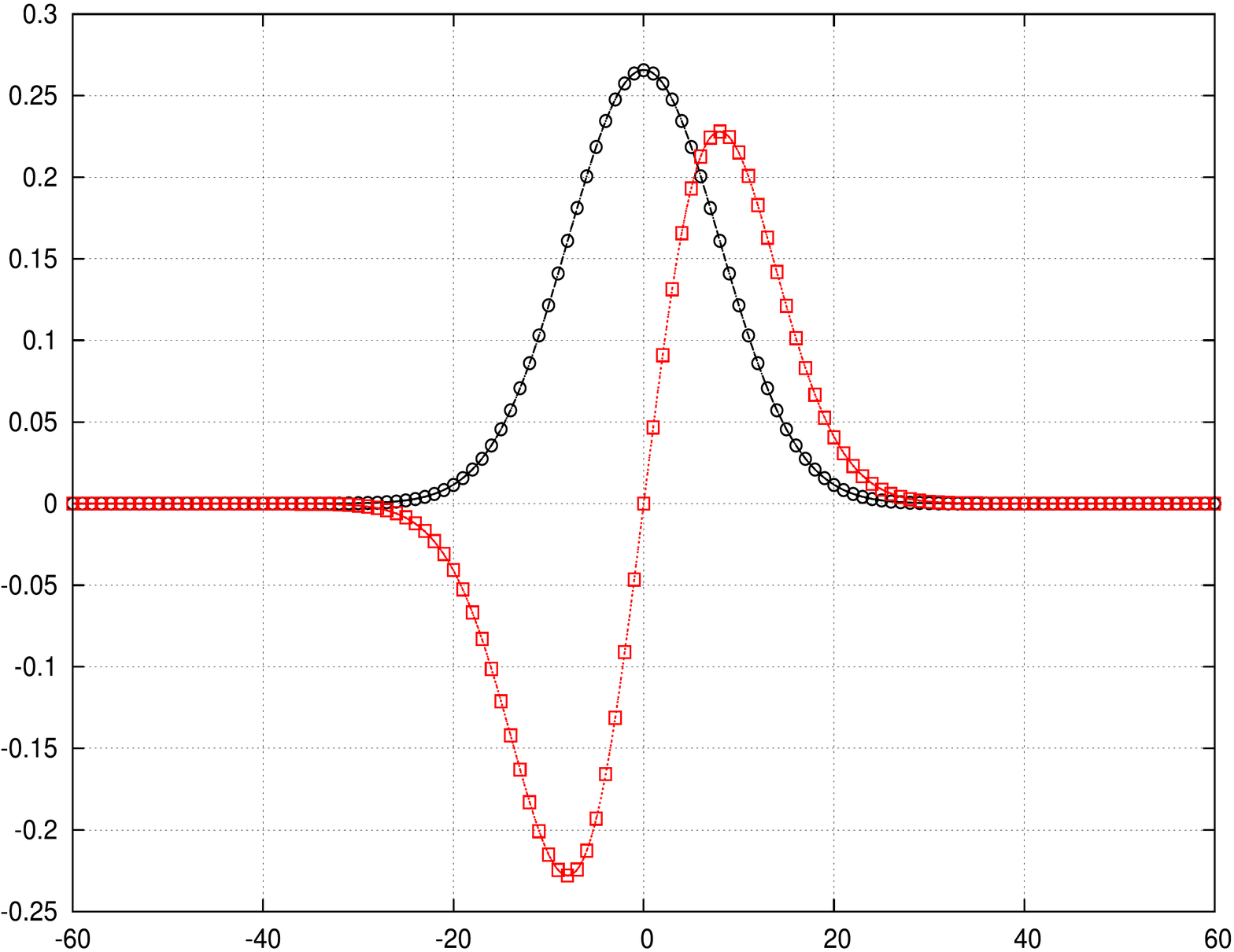}} 
\subfloat[][$n=2$ (black) and $n=3$ (red)]{\label{L100_f2}\includegraphics[height =4.7cm]{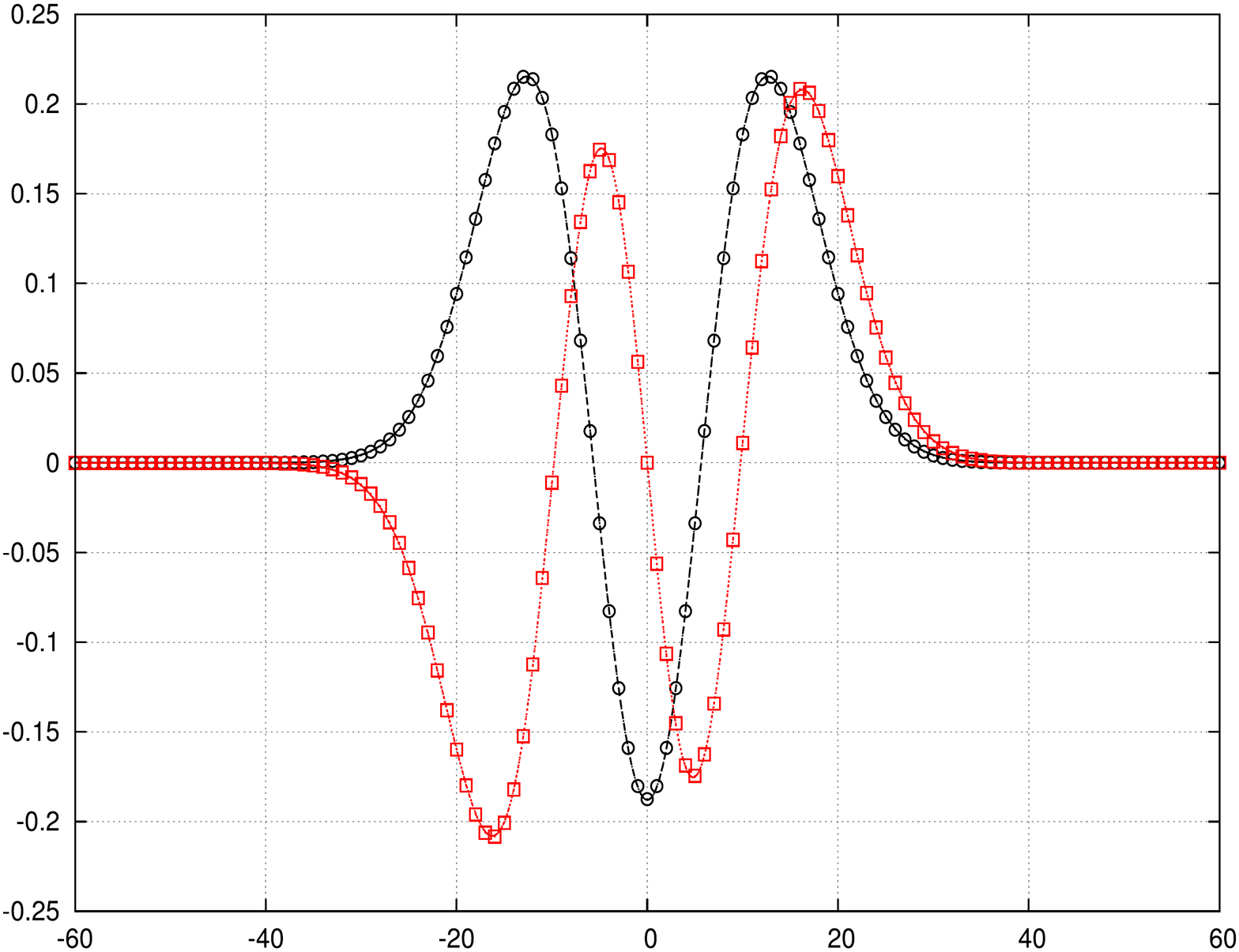}} 
\subfloat[][$n=4$ (black) and $n=5$ (red)]{\label{L100_f3}\includegraphics[height =4.7cm]{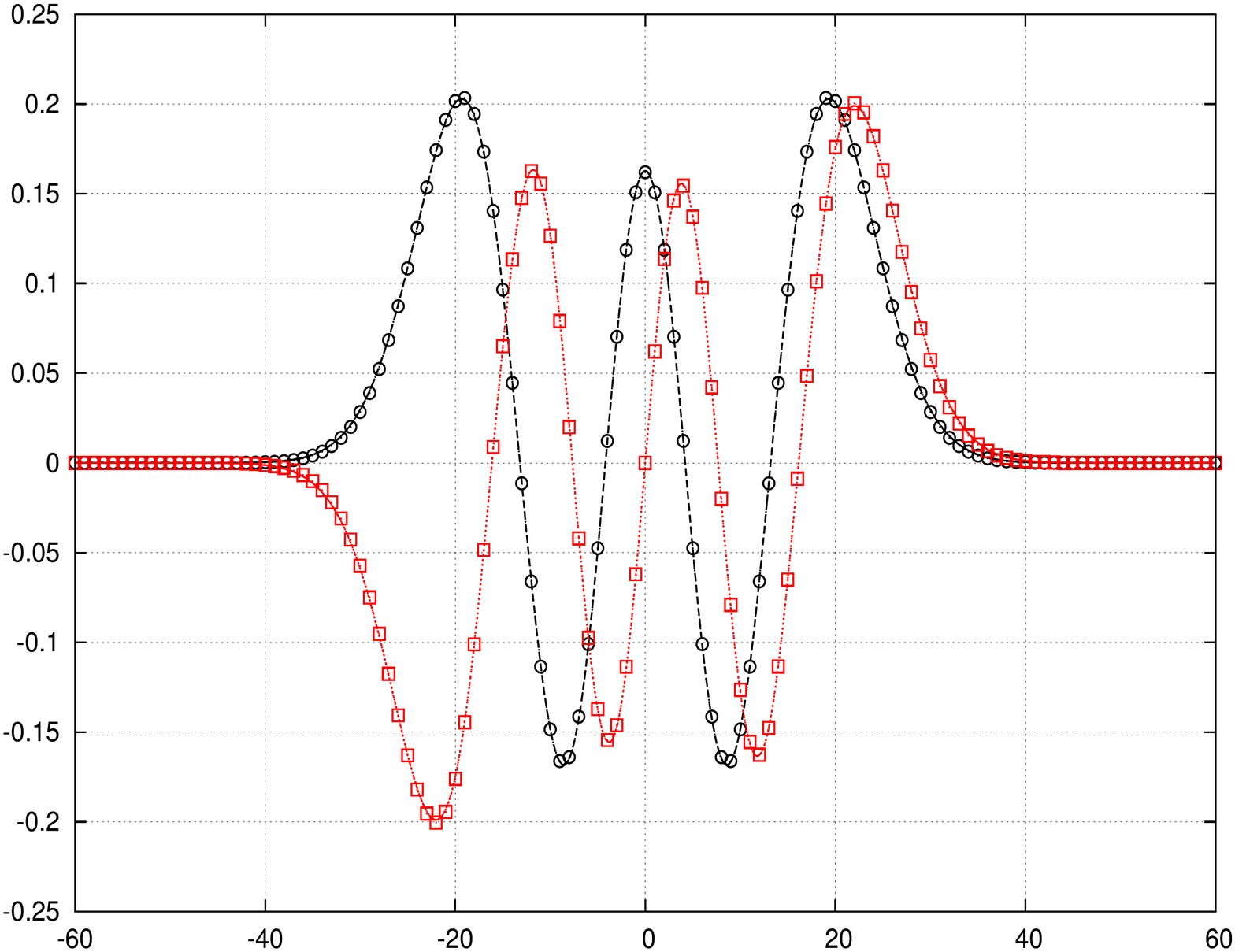}} \\
\subfloat[][$n=6$ (black) and $n=7$ (red)]{\label{L100_f4}\includegraphics[height =4.7cm]{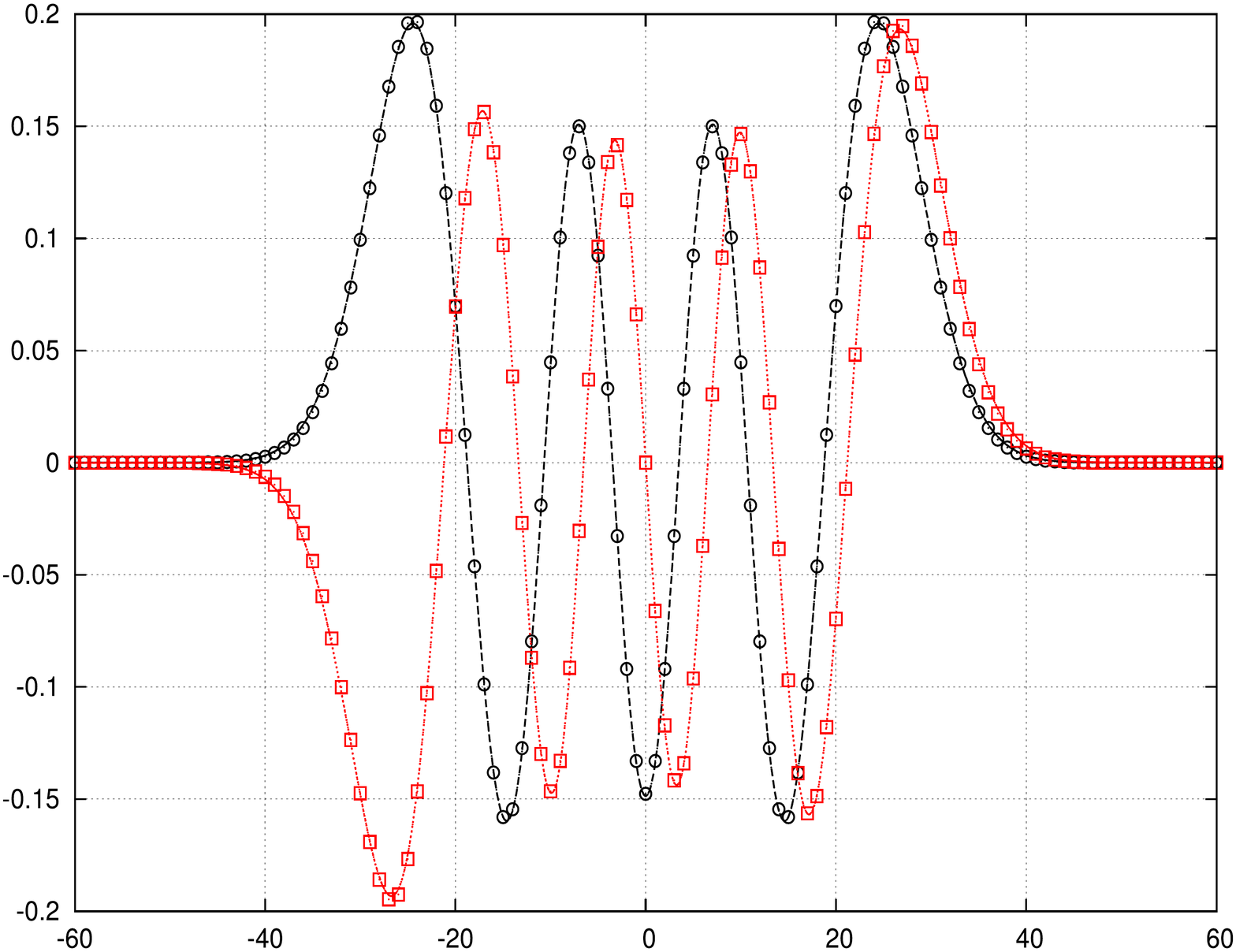}} 
\subfloat[][$n=8$ (black) and $n=9$ (red)]{\label{L100_f5}\includegraphics[height =4.7cm]{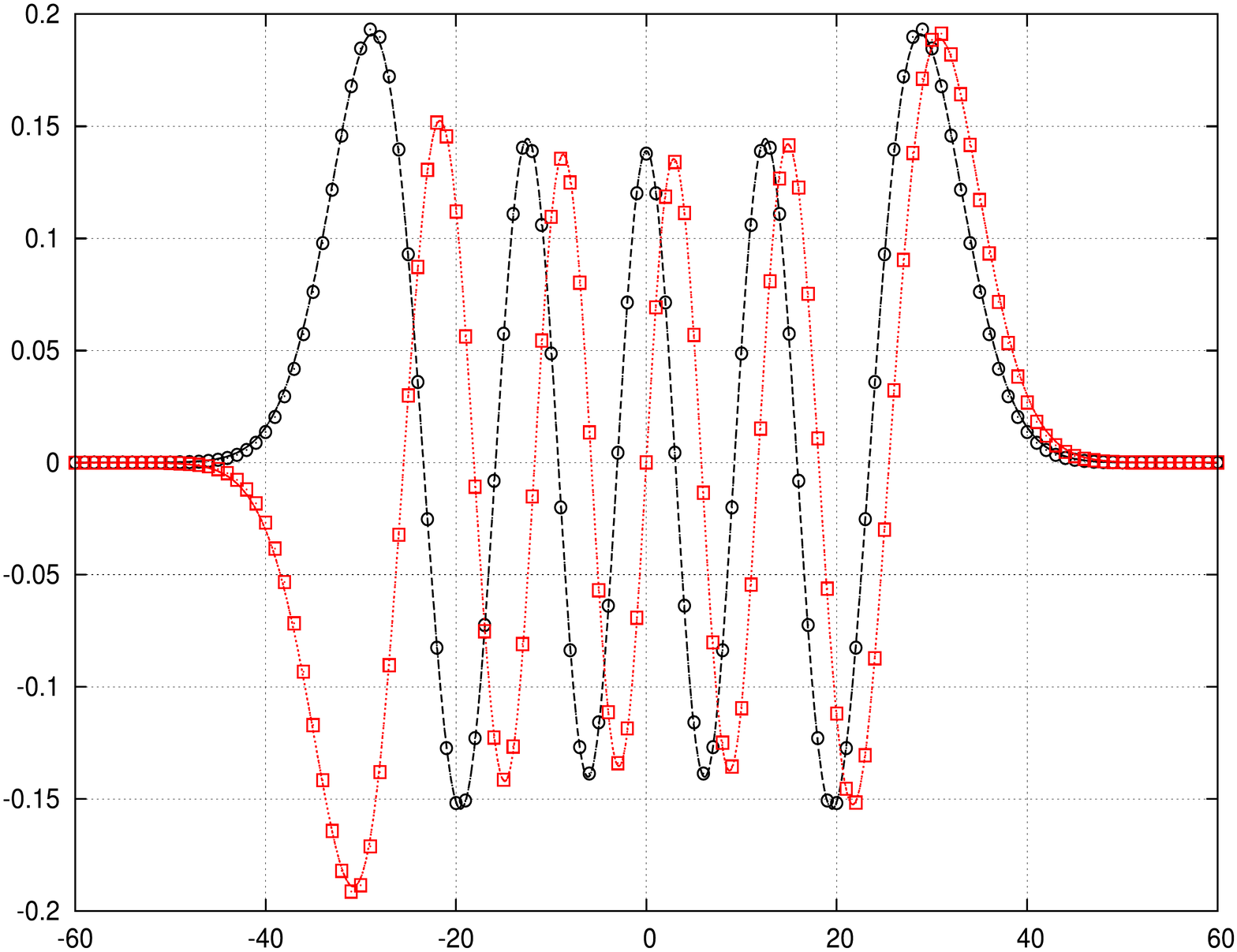}} 
\subfloat[][$n=10$ (black) and $n=11$ (red)]{\label{L100_f6}\includegraphics[height =4.7cm]{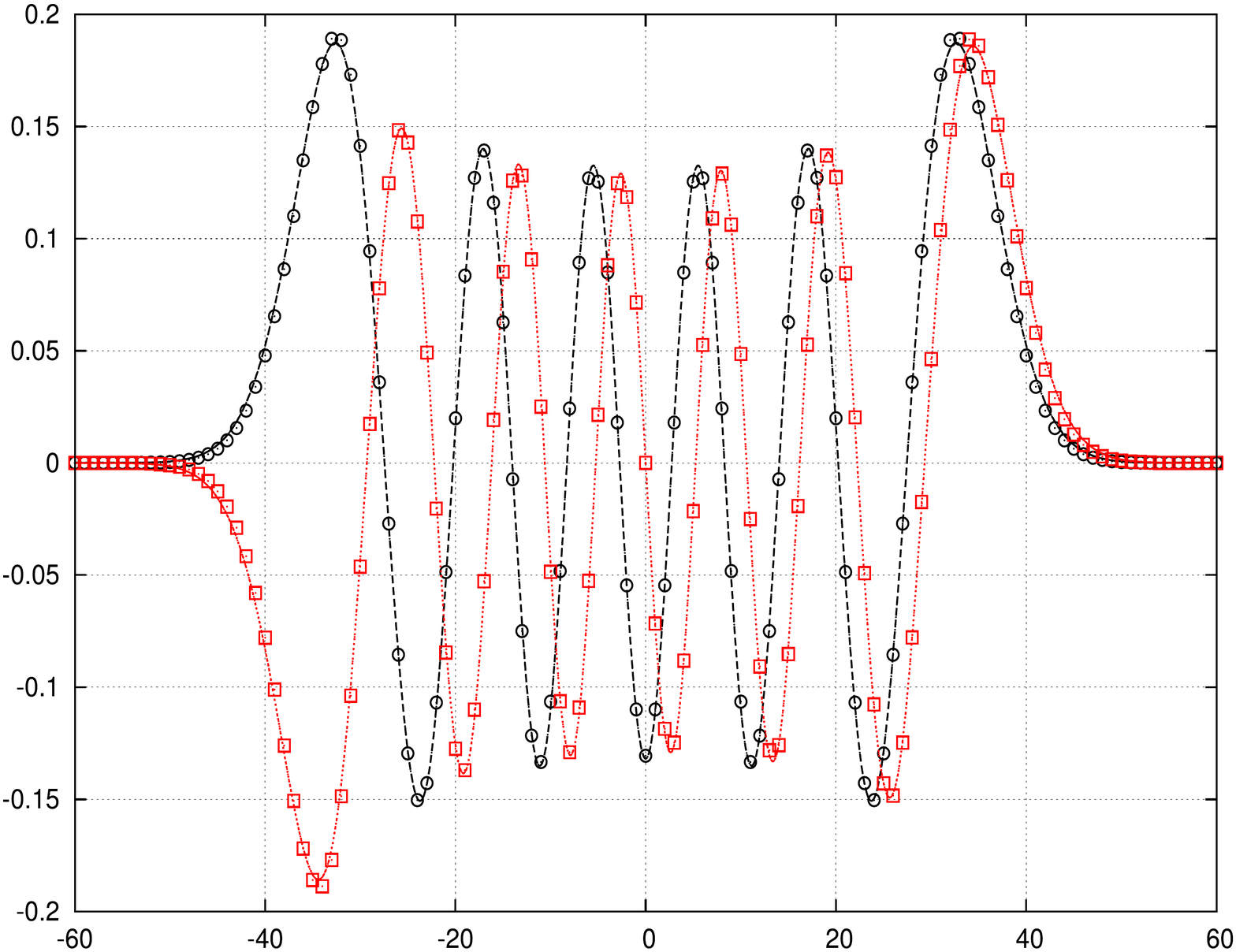}}
\caption{Comparing the first twelve eigenvectors $\phi_n(k)$ (black circles and red squares) with the scaled Hermite functions $\epsilon^{1/2}\psi_n(\epsilon k)$ (black and red lines). Here $N=401$, $\epsilon=\sqrt{2\pi /N}$ and the $x$-axis represents 
the index $k \in [-60,60]$.}
\label{fig_L100}
\end{figure}

The first six (respectively, twelve) eigenvectors ${\boldsymbol \phi}_n$ and the corresponding scaled Hermite functions are shown on 
Figure \ref{fig_L10} for $N=41$  (respectively, Figure  \ref{fig_L100} for $N=401$). The results presented on Figure \ref{fig_L100} 
suggest that the convergence  in \eqref{convergence_phi_psi} is not restricted to the first eight eigenvectors, and that it should hold 
for all of them.  Thus we formulate

\vspace{0.25cm}
\noindent
{\bf Conjecture 1:} For all $n \ge 0$ we have 
\begin{equation}
\max_{-2L \le k \le 2L}  |\epsilon^{-1/2} \boldsymbol\phi_n(k)-\psi_n(\epsilon k)| \to 0 \;
{\textnormal{ as }} N\to +\infty.
\end{equation}

 Our proof of Theorem \ref{thm_convergence} -- which is presented in Section \ref{section_proofs} -- 
is essentially a proof by verification, and it does not provide an intuitive reason as to why this convergence should hold. 
Our approach would probably work for $n \ge 8$, though 
larger values of $n$ would require a prohibitive amount of algebraic computations. Thus our 
method of proof of Theorem \ref{thm_convergence} is probably not the right way for proving 
Conjecture 1. We hope that some other properties of the vectors ${\mathbf u}_n$ and ${\mathbf v}_n$ (which we may have 
overlooked) will lead to a simple and insightful proof of this conjecture.

Figure \ref{fig_L40} provides a ``bird's eye view" of the eigenbasis ${\boldsymbol \phi_n}$ and compares it with Hermite functions. 
In particular, on Figure \ref{L40_f3} we show the values of $n$ and $k$ where $\phi_n(k)$ is either positive, or negative or equal to zero. 
From these results it seems that the vectors ${\boldsymbol \phi}_n$ have the same number of zero crossings as the corresponding 
Hermite functions $\psi_n$. Let us make this statement precise.  
We define the number of zero crossings of a vector ${\mathbf a} \in \r^N$ as the number of pairs $(i,j)$ 
such that (i) $i,j \in I_N$, (ii) $i<j$, (iii) $a(i)a(j)<0$ and (iv) $a(k)=0$ for all $i<k<j$.
Based on numerical evidence obtained from high-precision computations for many different values of $N$ 
(see Figure \ref{L40_f3} for $N=161$) we arrive at

\vspace{0.25cm}
\noindent
{\bf Conjecture 2:} For $0\le n \le N-1$ the vector ${\boldsymbol \phi}_n$ has exactly $n$ zero crossings. 
\vspace{0.15cm}

\newpage

Figure  \ref{L40_f3} also illustrates the following result 
\begin{equation}\label{length_phin}
l(\boldsymbol \phi_n)=2L+2\lceil n/4 \rceil +1,
\end{equation}
which follows easily from \eqref{support_uv} and the above construction of ${\boldsymbol \phi}_n$. 
In particular, this result implies that about 25\% of all elements $\phi_n(k)$ of the eigenbasis are equal to zero. 
In view of Theorem \ref{thm_support} the following question naturally arises: are there other sets of eigenvectors 
$\{{\boldsymbol \varphi}_n\}_{0\le n \le N-1}$ 
of the DFT such that ${\boldsymbol \varphi}_n$ satisfies \eqref{length_phin} and has exactly $n$ zero crossings? In other words, is the 
eigenbasis ${\boldsymbol \phi_n}$ unique in a similar sense in which the vectors ${\mathbf u}_n$ and ${\mathbf v}_n$ are unique?

\begin{figure}
\centering
\subfloat[][Hermite functions $\epsilon^{1/2}\psi_n(\epsilon k)$]{\label{L40_f1}\includegraphics[height =5.7cm]{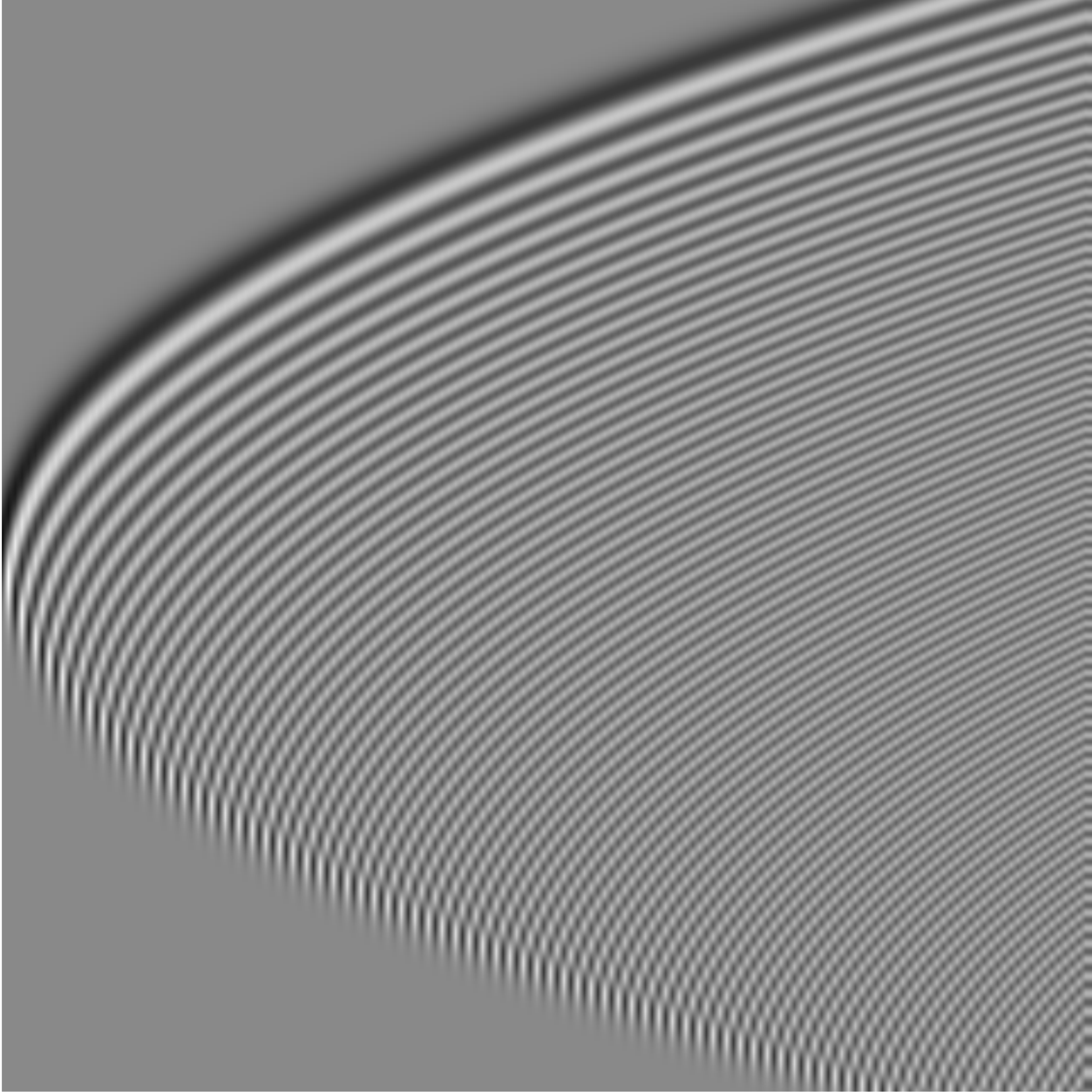}} 
 \;\;\;
 \subfloat[][Eigenvectors $\phi_n(k)$ ]{\label{L40_f2}\includegraphics[height =5.7cm]{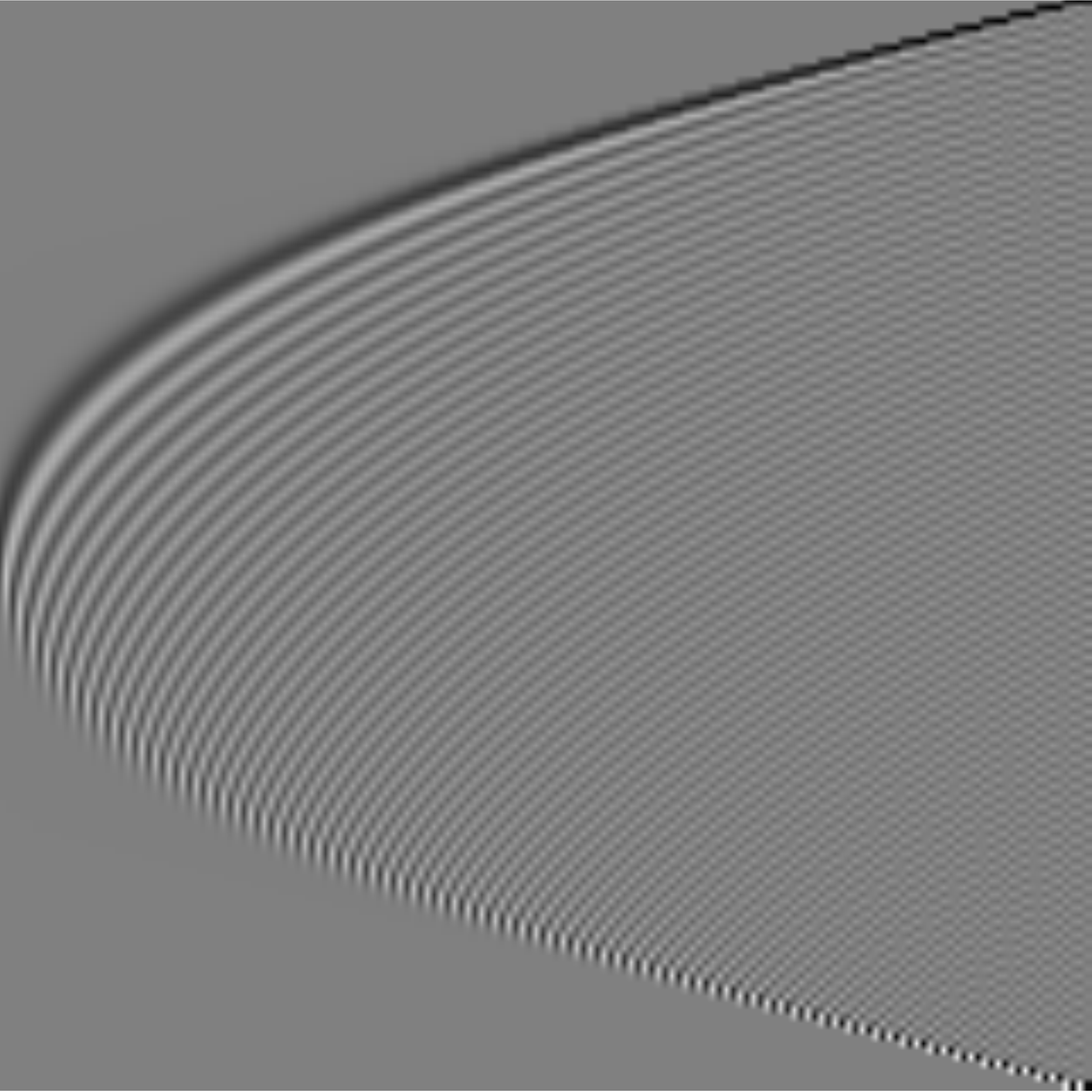}} 
\;\;\;
\subfloat[][The sign of $\phi_n(k)$]{\label{L40_f3}\includegraphics[height =5.7cm]{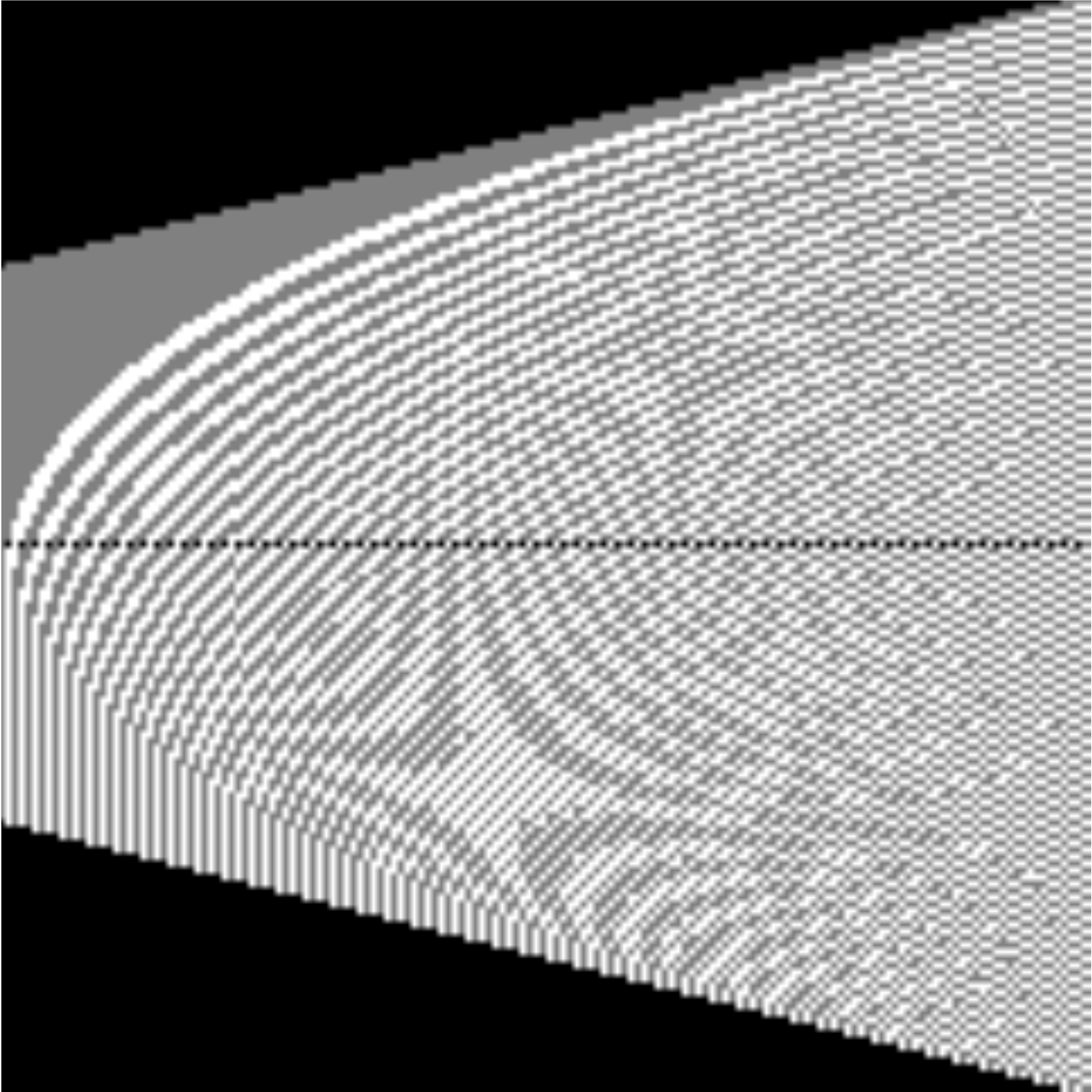}} 
\caption{Intensity plots of Hermite functions $\epsilon^{1/2}\psi_n(\epsilon k)$,  the eigenvectors $\phi_n(k)$ 
and the sign of the eigenvectors $\phi_n(k)$. Here $N=161$, $\epsilon=\sqrt{2\pi /N}$, the $x$-axis represents the index $n \in \{0,1,\dots,160\}$ 
and the $y$-axis represents the index $k \in \{-80,-79, \dots, 80\}$. Graph (c) shows the values where $\phi_n(k)=0$ (black pixels), $\phi_n(k)>0$ (grey pixels) and $\phi_n(k)<0$ (white pixels).}
\label{fig_L40}
\end{figure}


\section{Proofs}\label{section_proofs}


For a nonzero vector ${\mathbf a} \in \c^N$ we denote by $R({\mathbf a};z)$ the rational function 
\begin{equation}\label{def_Raz}
R({\mathbf a};z):=\sum\limits_{k \in I_N} a(k) z^k. 
\end{equation}
We also denote $\zeta:=\exp(-2\pi \i /N)$.
It is clear that ${\mathbf b}={\mathcal F} {\mathbf a}$ if and only if $b(k)=N^{-1/2}R({\mathbf a};\zeta^k)$ for all
$k\in I_N$. The following simple observation will play the main role in the proof of Theorem \ref{thm_support}: 
\begin{equation}\label{Raz_property}
\textnormal{{\it The function $R({\mathbf a};z)$ has exactly $l({\mathbf a})-1$ zeros in 
${\c}\setminus\{0\}$ (counting their multiplicity)}.}
\end{equation}
 To establish this fact we define $j=\min\{ k \in I_N \;: a(k)\ne 0\}$ and note 
that $z^{-j} R({\mathbf a};z)$ is a polynomial of degree $l({\mathbf a})-1$. 

To give a preview of our method of the proof of Theorem \ref{thm_support}, let us prove Kong's result \eqref{Kong_Thm1} (see \cite[Theorem 1]{Kong_2008}) using the above property \eqref{Raz_property}.
 Assume that for some nonzero vector ${\mathbf a} \in \c^N$
we have $l({\mathbf a})=m$ and $l({\mathbf b})\le N-m$ where ${\mathbf b}={\mathcal F}{\mathbf a}$. Since $l({\mathbf b})\le N-m$ it follows 
that $b(k)=0$ for at least $m$ indices $k\in I_N$. In view of identity $b(k)=N^{-1/2}R({\mathbf a};\zeta^k)$, we see that 
the nonzero rational function $R({\mathbf a};z)$ has at least $m$ distinct zeros on the unit circle, which is impossible because of 
\eqref{Raz_property}.  

\vspace{0.25cm}
\noindent
{\bf Proof of Theorem \ref{thm_support}:}
Everywhere in this proof we will denote ${\mathbf b}={\mathcal F}{\mathbf a}$. When we say that a vector or a function is 
``unique" we will mean ``unique, up to multiplication by a nonzero constant". 

Let us first prove part (i) for even vectors. Assume that $N=2M+1$ for some $M\ge 1$, so that 
$I_N=\{-M,-M+1,\dots,M\}$. 
It is clear that there exists a unique even vector ${\mathbf a} \in {\mathbf R}^N$ of length one 
 (its only nonzero element is $a(0)$). The corresponding DFT vector ${\mathbf b}$ has elements
 $b(k)=a(0)N^{-1/2}$ and has length $N$. Having dealt with this special case, we now  
assume that $l({\mathbf a})=2m+1$ for some $1\le m \le M-1$. If we impose the constraint $l({\mathbf a})+l({\mathbf b})=N+1$, then 
$l({\mathbf b})=2M-2m+1$ and we must have $R({\mathbf a};\zeta^k)=b(k)=0$ for $M-m+1 \le |k|\le M$. This gives us $2m$ zeros 
of $R({\mathbf a};z)$, and since $z^m R({\mathbf a};z)$ is a polynomial of degree $2m$ we can uniquely
 identify 
$$
R({\mathbf a};z)=z^{-m} \prod\limits_{k=M-m+1}^M (z-\zeta^k)(z-\zeta^{-k}), 
$$ 
which in turn uniquely identifies the vector ${\mathbf a}$. Thus we have shown that 
for every $1\le m \le M-1$ there exists a unique even vector ${\mathbf a}$ of length $2m+1$ satisfying  
$l({\mathbf a})+l({\mathbf b})=N+1$. Since there is also exactly one even vector of length one and exactly one
even vector of length $N$ satisfying the same property (these were discussed above), we conclude that 
the total number of such vectors is $M+1$.

Consider next an odd vector ${\mathbf a} \in {\mathbb R}^N$ of length $2m+1$. Note that there are no odd vectors of length one, thus 
we can assume that $1\le m\le M$.
Since ${\mathbf a}$ is an odd vector, the function $R({\mathbf a};z)$ can be written in the form
$$
R({\mathbf a};z)=\sum\limits_{k=1}^m a(k) (z^{k}-z^{-k}).
$$
In particular, we see that $R({\mathbf a};\pm 1)=0$. Assuming that 
 $l({\mathbf a})+l({\mathbf b})=N+3$ we obtain $l({\mathbf b})=2M-2m+3$. Since ${\mathbf b}$ is an odd vector 
 of length $2M-2m+3$ we must have $R({\mathbf a};\zeta^k)=b(k)=0$ for $k=0$ and $M-m+2 \le |k|\le M$.
  This gives us $2m-1$ distinct zeros of $R({\mathbf a};z)$. Taking into account the zero at $z=-1$ and using 
  the fact that $z^{m} R({\mathbf a};z)$ is a polynomial of degree $2m$, we can uniquely identify
$$
R({\mathbf a};z)=z^{-m} (z-1)(z+1) \prod\limits_{k=M-m+2}^M (z-\zeta^k)(z-\zeta^{-k}).
$$ 
Thus we have proved that for every $1\le m \le M$ there exists a unique odd vector ${\mathbf a}$ of length $2m+1$ satisfying  
$l({\mathbf a})+l({\mathcal F}{\mathbf a})=N+3$. From our proof it should be clear that there can not exist an odd vector satisfying 
$l({\mathbf a})+l({\mathcal F}{\mathbf a})=N+1$, because in this case the function $R({\mathbf a};z)$ would 
have $2m+2$ zeros, which is impossible due to \eqref{def_Raz}. This completes the proof of part (i). 

Next, let us prove part (ii) for even vectors. Assume that $N=2M$ and that ${\mathbf a} \in {\mathbb R}^N$ is a nonzero 
even vector.  
In this case 
$$
R({\mathbf a};z)=a(0)+\sum\limits_{k=1}^{M-1} a(k) (z^k+z^{-k})+a(M)z^M,
$$
and assuming that $a(M)=0$ we see that this function satisfies $R({\mathbf a};z)=R({\mathbf a};1/z)$. This identity implies
\begin{equation}\label{R_prime}
\frac{\d }{\d z}R({\mathbf a};z) \bigg \vert_{z=\pm 1}=0.
\end{equation} 
There is a unique even vector of length one satisfying $a(M)=0$ (its only nonzero elements is $a(0)$), however for this vector 
we have $b(M)=a(0)N^{-1/2}\neq 0$, thus we have to 
exclude this vector from consideration and we can assume that $l({\mathbf a})=2m+1$ with $1\le m \le M-1$. 
If $l({\mathbf a})+l(\mathbf b)=N+2$ then $l(\mathbf b)=2M-2m+1$. 
Since ${\mathbf b}$ is also an even vector with $b(M)=0$, the equality $l(\mathbf b)=2M-2m+1$ implies that 
$R({\mathbf a};\zeta^k)=b(k)=0$ for $k=M$ and 
for $M-m+1\le |k|\le M-1$. Note that $\zeta^M=-1$, thus $R({\mathbf a};z)$ has a zero at $z=-1$, and from condition \eqref{R_prime} we see that this zero must have multiplicity two. 
Thus we have determined all $2m$ zeros of the polynomial $z^{m}R({\mathbf a};z)$ of degree $2m$, which allows us to  identify
uniquely
$$
R({\mathbf a};z)=z^{-m} (z+1)^2 \prod\limits_{k=M-m+1}^{M-1} (z-\zeta^k)(z-\zeta^{-k}).
$$   
To summarize: we have proved that for every $1\le m \le M-1$ there exists a unique even vector ${\mathbf a}$ of length $2m+1$ satisfying  
$a(M)=b(M)=0$ and $l({\mathbf a})+l({\mathbf b})=N+2$. From our proof it is clear that there can not exist an even vector with the same properties satisfying $l({\mathbf a})+l({\mathbf b})=N+1$, because in this case the function $R({\mathbf a};z)$ would have $2m+1$ zeros, which is impossible due to \eqref{Raz_property}. 
This completes the proof of part (ii) for even vectors.

The proof of part (ii) for odd vectors follows the same steps and is left to the reader. 
\qed

The proof of Theorem \ref{theorem_4N_plus1_n1} (and the derivation of the corresponding formulas in the Appendix A) is based on the following lemma. We 
emphasize that this result is true for {\it all} $N\ge 1$ (not only those satisfying $N \equiv 1 (\textnormal {mod }  4)$). 

\begin{lemma}\label{main_lemma}
For all $N \ge 1$ and $m,n \in \{0,1,\dots,N-1\}$ we have
\begin{equation}\label{eqn_main_lemma}
\sum\limits_{k=0}^{N-1-n} S(n+k) S(N-k-1) e^{-\frac{\pi \i}{2N} (2k+n+1-N)(2m-n)} =S(n)S(N+m-n-1)S(N-m-1).
\end{equation}
\end{lemma}
\begin{proof}
Let us define the $q$-Pochhammer symbol as 
\begin{equation*}
(z;q)_{j}=\prod\limits_{k=0}^{j-1} (1-aq^k),  \;\;\; \textnormal{ for } \; j\ge 1,  
\end{equation*}
and $(z;q)_0=1$. We begin with the $q$-Binomial Theorem
\begin{equation*}
(xq;q)_n=\sum\limits_{m=0}^n \frac{(q;q)_n}{(q;q)_m (q;q)_{n-m}} q^{m(m+1)/2} (-1)^m x^m.
\end{equation*}
When $x=q^k$, the above equality gives us 
\begin{equation}\label{lem1_proof0}
\frac{(q;q)_{n+k}}{(q;q)_k}=\sum\limits_{m=0}^n \frac{(q;q)_n}{(q;q)_m (q;q)_{n-m}} q^{m(m+1)/2} (-1)^m q^{km},
\end{equation}
where we have also used the identity $(q^{k+1};q)_n=(q;q)_{n+k}/(q;q)_k$.
Next, we set $q=\exp(2\pi \i/N)$ in \eqref{lem1_proof0}, we check that this choice of $q$ implies
\begin{equation*}
(q;q)_k=S(k)  e^{-\pi\i k (N-k-1)/(2N)}, \;\;\; \textnormal{ for } \; k\ge 0, 
\end{equation*}
and we obtain 
$$
\frac{S(n+k)}{S(k)} e^{-\pi \i n(N-2k-n-1)/(2N)}=\sum\limits_{m=0}^n \left[ \frac{S(n)}{S(m) S(n-m)}
(-1)^m e^{\pi \i m(n+1)/N} \right] e^{2\pi \i km/N}.
$$
The expression in the right-hand side is the inverse Fourier transform of the sequence given in square brackets. 
Applying the discrete Fourier transform results in
$$
\frac{S(n)}{S(m) S(n-m)}
(-1)^m e^{\pi \i m(n+1)/N}=N^{-1}\sum\limits_{k=0}^{N-1} 
\frac{S(n+k)}{S(k)} e^{-\pi \i n(N-2k-n-1)/(2N)} e^{-2\pi \i km/N}. 
$$
Formula \eqref{eqn_main_lemma} follows by simplifying the above expression and applying the identity
\eqref{sk_properties2}. 
\end{proof}

\vspace{0.25cm}
\noindent
{\bf Proof of Theorem \ref{theorem_4N_plus1_n1}:} 
In order to prove formula \eqref{eqn_Fu_n}, we set $N=4L+1$  and replace variables 
$k$, $m$ and $n$  in the equation \eqref{eqn_main_lemma} with $k \mapsto L-n+k$, $m \mapsto L+n+l$ and $n \mapsto 2L+2n$. This gives us the following identity 
$$
\sum\limits_{k=-L+n}^{L-n} S(3L+n+k)S(3L+n-k) e^{-2\pi \i kl/N} = S(2L+2n) S(3L-n+l)S(3L-n-l),
$$
which is valid for  $-L+n \le k \le L-n$, $-L-n\le l \le 3L-n$ and $-L \le n \le L$.  
Extending the above identity by periodicity for the values of $l$ in the range $-2L \le l \le 2L$ gives us the desired result \eqref{eqn_Fu_n}.
 
The fact that these vectors are the same ones that were discussed in Theorem \ref{thm_support}(i) follows from \eqref{eqn_Fu_n} and the fact that $l({\textbf u}_n)=2L-2n+1$. In order to prove that these vectors are linearly independent, let us assume to the contrary that
\begin{equation}\label{linear_combination}
a_1{\mathbf u}_{i_1}+a_2{\mathbf u}_{i_2}+\dots+a_k{\mathbf u}_{i_k}={\mathbf 0},
\end{equation}
where ${\mathbf 0}$ denotes the zero vector, $a_j\neq 0$ for all $1\le j \le k$ and $-L \le i_1<i_2<\dots<i_k\le L$. Note that, by construction,
$u_n(k)=0$ for $L-n+1\le |k| \le 2L$  and $u_n(k)>0$ for $|k|\le L-n$. Therefore we have $u_{i_k}(L-i_k)>0$ while 
$u_{i_l}(L-i_k)=0$ for all other indices $i_l$. Equation \eqref{linear_combination} then forces $a_k$ to be zero and we arrive at a contradiction. 
 
To prove formula \eqref{eqn_Fv_n}, we use \eqref{def_un} and and check that 
\begin{align*}
&4 \sin(\pi (3L+n+k)/N) \sin(\pi (3L+n+1+k)/N) u_{n-1}(k)=u_n(k+1),\\
&4 \sin(\pi (3L+n-k)/N) \sin(\pi (3L+n+1-k)/N) u_{n-1}(k)=u_n(k-1).
\end{align*}
Subtracting the above two equations and simplifying the result gives us formula \eqref{vn_k1_minus_k2}. 
Applying the discrete Fourier transform to both sides of \eqref{vn_k1_minus_k2} we obtain
\begin{align}\label{lemma_2_proof1}
&-4 \sin(\pi(2L+2n)/N) ({\mathcal F}{\mathbf v}_{n-1})(l)=\sum\limits_{k=-2L}^{2L} e^{-2\pi \i k l/N} (u_n(k+1)-u_n(k-1))\\ \nonumber
&\qquad \qquad=
2 \i \sin(2\pi l/N) \sum\limits_{k=-2L}^{2L} e^{-2\pi \i k l/N} u_n(k)=
2 \i \sin(2\pi l/N) ({\mathcal F}{\mathbf u}_{n})(l). 
\end{align}
In deriving the above identity we have used the fact that $u_n(\pm 2L)=0$ for $n>-L$. Formula \eqref{eqn_Fv_n}
follows from \eqref{eqn_Fu_n} and \eqref{lemma_2_proof1}. The linear independence of the vectors ${\mathbf v}_n$ can be established by  the same argument as we used for ${\mathbf u}_n$.
\qed

The following result will be needed for the proof of Theorem \ref{thm_convergence}. 
We recall the definition of Catalan's constant
\begin{equation*}
G:=- \int_0^{\pi/2} \ln (2 \sin(t/2)) \d t=0.915965594...
\end{equation*}

\begin{lemma}\label{lemma_v0}
There exists an absolute constant $C>0$ such that for all $N$ large enough
and all $k \in [-N^{2/3},N^{2/3}]\cap {\mathbb Z}$ we have
\begin{equation}\label{main_estimate}
u_0(k)=2^{-1/4} N e^{NG/\pi-\pi k^2/N}(1+{\mathcal E})
\end{equation}
and $|{\mathcal E}|<C N^{-1/3}$. 
\end{lemma}
\begin{proof}
First of all, we use property \eqref{sk_properties2}  and check that 
\begin{equation}\label{u_0_s3l_sL}
u_0(k)=S(3L+k)S(3L-k)=N \frac{S(3L+k)}{S(L+k)}. 
\end{equation}
Let us denote $f(x):=\ln(2 \sin(\pi x/N))$ for $0<x<N$. Applying the Euler-Maclaurin summation formula we obtain
\begin{align}\label{lemma_v0_1} 
\nonumber
\ln(S(3L+k)/S(L+k))=\sum\limits_{j=L+k+1}^{3L+k} f(j)&=\int_{L+k}^{3L+k} f(x) \d x+(f(3L+k)-f(L+k))/2
\\ 
&+(f'(3L+k)-f'(L+k))/12+{\mathcal R},
\end{align}
where the remainder term has an upper bound
$$
|{\mathcal R}|<\frac{1}{12} \int_{L+k}^{3L+k} |f''(x)| \d x.
$$
We compute the first term in the right-hand side in \eqref{lemma_v0_1}  
$$
I_1:=\int_{L+k}^{3L+k} f(x) \d x=-\frac{N}{2\pi}({\textrm{Cl}}_2(2 \pi (3L+k)/N)-{\textrm{Cl}}_2(2 \pi (L+k)/N)),
$$
where the Clausen function  is defined by
$$
{\textrm{Cl}}_2(t):=-\int_0^t \ln | 2 \sin(u/2)| \d u.
$$
The function ${\textrm{Cl}}_2(t)$ is odd and periodic with period $2\pi$ (see \cite[Chapter 4]{Lewin1981}), which gives us
$$
I_1=\frac{N}{2\pi} ({\textrm{Cl}}_2(2 \pi (L+1-k)/N)+{\textrm{Cl}}_2(2 \pi (L+k)/N)).
$$
The Clausen function can be written as Taylor series ${\textrm{Cl}}_2(\pi/2+z)=\sum_{j\ge 0} a_j z^j$, where $a_0=G$, $a_1=-\ln(2)/2$, $a_2=-1/4$ and the series converges if $|z|<\pi/2$. Using this result we obtain
\begin{align*}
I_1&=\sum\limits_{j\ge 0} a_j (2\pi/N)^{j-1} ((3/4-k)^j+(k-1/4)^j)=NG/\pi-\ln(2)/4-\frac{\pi}{2N} (2k^2-2k+5/8)\\
&+\left(\frac{\pi}{2N}\right)^2 (3k^2/2-3k/2+13/32)
+\sum\limits_{j\ge 4} a_j (2\pi/N)^{j-1} ((3/4-k)^j+(k-1/4)^j).
\end{align*}
Assuming that $|k|<N^{2/3}$ we have
\begin{align*}
I_1&=NG/\pi-\ln(2)/4-\pi k^2/N + O(N^{-1/3}) 
+\sum\limits_{j\ge 4} a_j (2\pi/N)^{j-1} ((3/4-k)^j+(k-1/4)^j).
\end{align*}
When $N>(4\pi)^3$, each term in the series in the right-hand side of the above formula can be estimated as follows
\begin{align*}
\big \vert (2\pi/N)^{j-1} ((3/4-k)^j+(k-1/4)^j) \big \vert &<2 (N^{2/3}+1)^j (2\pi)^{j-1} N^{1-j}
\\&<
(4\pi)^j N^{1-j/3}<(4\pi)^4 N^{-1/3},
\end{align*}
for all $j\ge 4$. 
Combining all the above two formulas and using the fact that the series $\sum_{j\ge 0} |a_j|$ converges we arrive at our final estimate 
of $I_1$, which is valid for all $N>(4\pi)^3$ and all $|k|<N^{2/3}$:
$$
I_1=NG/\pi-\ln(2)/4-\pi k^2/N +O(N^{-1/3}).
$$

Dealing with the remaining terms in the right-hand side of \eqref{lemma_v0_1}  is much easier. We check that the second term  satisfies
\begin{align*}
I_2:=(f(3L+k)-f(L+k))/2=\frac{1}{2} \ln\left(\frac{\sin(\frac{\pi}{4}+\pi(3/4-k)/N)}{\sin(\frac{\pi}{4}+\pi(k-1/4)/N)}\right)=
\frac{1}{2} \ln(1+O(N^{-1/3}))=O(N^{-1/3}) 
\end{align*}
Similarly, the third term can be estimated as
\begin{align*}
I_3:=(f'(3L+k)-f'(L+k))/12=\frac{\pi}{12 N} (\cot(\pi/4+\pi(3/4-k)/N)+\cot(\pi/4+\pi(k-1/4)/N))=O(N^{-1})
\end{align*}
Finally, the remainder can be bounded as follows
\begin{align*}
|{\mathcal R}|&<\frac{1}{12} \int_{L+k}^{3L+k} |f''(x)| \d x=\frac{\pi^2}{12 N^2} \int_{L+k}^{3L+k} \csc(\pi x/N)^2 \d x \\
&\le 
\frac{\pi^2}{12 N^2} \times (N/2) \times  \max \{ \csc(\pi x/N)^2 \; : \; N/4-N^{2/3} \le x \le 3N/4+N^{2/3} \}= O(N^{-1}). 
\end{align*}
Combining \eqref{u_0_s3l_sL} with \eqref{lemma_v0_1} and the above four asymptotic results ends the proof of \eqref{main_estimate}. 
\end{proof}

\vspace{0.25cm}
\noindent
{\bf Proof of Theorem \ref{thm_convergence}:}
We recall that $\epsilon=\sqrt{2\pi/N}$ and we define
$\eta:=2^{-1/4} N \exp(NG/\pi)$ and
$$
J_N:=\{ k \in {\mathbb Z} \; : \; |\epsilon k|<2N^{1/16}\}. 
$$

Our first goal is to prove the following upper bound: there exists an absolute constant $C_1>0$ such that 
\begin{equation}\label{estimate_un}
0\le u_n(k)<C_1 \eta \exp(- N^{1/8}), \;\; \textnormal{ for } \; |n| \le 2  \; 
\textnormal{ and } \; k \in I_N \setminus J_N. 
\end{equation}
In order to establish \eqref{estimate_un} we note that 
$$
u_n(k+1)=\frac{\sin(\pi(L-n-k)/N)}{\sin(\pi (L-n+k+1))} u_n(k),
$$ 
which shows that $u_n(k+1)\le u_n(k)$ for $k\ge 0$. By symmetry we also have $u_n(k-1)\le u_n(k)$ for $k\le 0$. 
Lemma \ref{lemma_v0}  implies that when $\epsilon k$ is close to $2N^{1/16}$ then 
$u_0(k)<2\eta \exp(-2N^{1/8})$. Since $u_n(k)$ decreases as $|k|$ increases,  this bound 
holds true for all $k \in I_N \setminus J_N$, and we have proved \eqref{estimate_un} for $n=0$. 
For other values of $n$ we use the same method coupled with the identity
\begin{equation}\label{identity_u_n_u_n1}
u_{n+1}(k)=2 \left(\cos(2\pi k /N)-\cos(2\pi (L-n)/N)\right) u_n(k), 
\end{equation}
which follows at once from \eqref{u_n_product_formula}. 

Next, we introduce $t_k:=2\sin(\epsilon^2 k/2)/\epsilon$.  Note that as $\epsilon\to 0$ we have 
$t_k=\epsilon k + O(\epsilon^5 k^3)$. 
 Using the double-angle formula for the cosine function  we rewrite
\eqref{identity_u_n_u_n1} in the form
\begin{align}\label{identity_u_n_u_n1_form2}
\nonumber
u_{n+1}(k)&=(2-2\sin(\epsilon^2 (n+1/4))-\epsilon^2 t_k^2) u_n(k)\\ 
&=(2-\epsilon^2 (t_k^2+2n+1/2)+ \epsilon^6 (n+1/4)^3/3+O(\epsilon^{10}) ) u_n(k). 
\end{align}
We emphasize that the implied constant in the error term $O(\epsilon^{10})$ does not depend on $k$. 
Applying \eqref{identity_u_n_u_n1_form2} we obtain 
\begin{align}
\label{eqn_u0_u1}
u_0(k)&=(2- \epsilon^2 (t_k^2-3/2)+O(\epsilon^6))u_{-1}(k), \\ \nonumber
u_1(k)&=(4-\epsilon^2 (4t_k^2-2)+\epsilon^4(t_k^4-t_k^2-3/4)+O(\epsilon^6))u_{-1}(k). 
\end{align}
The identity $S(2L)=\sqrt{N}$ coupled with \eqref{def_sk} give us
\begin{align}\label{asymptotics_S2L_1}
N^{-1/2} S(2L+1)&=2 \sin(\pi (2L+1)/N)=2 \cos(\epsilon^2/4)=2+O(\epsilon^4),\\ \nonumber
N^{-1/2} S(2L+2)&=4 \sin(\pi (2L+1)/N)\sin(\pi (2L+2)/N)=4 \cos(\epsilon^2/4)\cos(3\epsilon^2/4)=4-5\epsilon^4/4 +O(\epsilon^8).
\end{align}
Using formulas \eqref{eqn_u0_u1} and \eqref{asymptotics_S2L_1} and the fact that $v_n(k)=\sin(2\pi k /N) u_n(k)$ we arrive at the following result
\begin{align}\label{wxyz0}
\nonumber
w_0(k)&=2u_0(k)=(4- \epsilon^2 (2t_k^2-3)+O(\epsilon^6))u_{-1}(k), \\ \nonumber
x_0(k)&=v_0(k)+N^{-1/2} S(2L+1) v_{-1}(k)=(4+\epsilon^2 (3/2 - t_k^2) + O(\epsilon^4)) \sin(2\pi k /N)u_{-1}(k), \\
y_0(k)&=-u_1(k)+N^{-1/2} S(2L+2) u_{-1}(k)=(\epsilon^2 (4t_k^2-2)-\epsilon^4 (t_k^4-t_k^2+1/2)+O(\epsilon^6)) u_{-1}(k), \\
\nonumber
z_0(k)&=-v_0(k)+N^{-1/2} S(2L+1) v_{-1}(k)=(\epsilon^2 (t_k^2-3/2)+O(\epsilon^4)) \sin(2\pi k/N) u_{-1}(k). 
\end{align}
Again, in all of the above formulas the implied constants in the error terms $O(\epsilon^{j})$ do not depend on $k$. Now we restrict 
$k$ to the interval $J_N$. For $k\in J_N$ we have 
$|t_k|<4N^{1/16}$, thus $\epsilon t_k^4=O(N^{-1/4})$.  We also check that 
$\sin(2\pi k/N)=\epsilon t_k +O((\epsilon t_k)^3))$ and 
\begin{equation}\label{eqn_u_minus1}
u_{-1}(k)=\frac{\eta}{2} \exp(- (\epsilon k)^2/2)(1+O(\epsilon^{2/3})), \;\;\;  k \in J_N, 
\end{equation}
which follows from \eqref{identity_u_n_u_n1_form2} and 
Lemma \ref{lemma_v0}. Combining all of the above results with the following formulas for Hermite polynomials
$$
H_0(x)=1, \;\;\; H_1(x)=2x, \;\;\; H_2(x)=4x^2-2, \;\;\; H_3(x)=8x^3-12x,
$$
we conclude that for $k \in J_N$ 
\begin{align}\label{eqn_wxyz0_H}
\nonumber
w_0(k)&=2 \eta H_0(t_k) \exp(- (\epsilon k)^2/2) (1+O(\epsilon^{2/3})), \\ \nonumber
x_0(k)&=\epsilon \eta H_1(t_k) \exp(- (\epsilon k)^2/2) (1+O(\epsilon^{2/3})), \\
y_0(k)&=\epsilon^2 (\eta/2) H_2(t_k) \exp(- (\epsilon k)^2/2) (1+O(\epsilon^{2/3})), \\ \nonumber
z_0(k)&=\epsilon^3 (\eta/16) H_3(t_k)  \exp(- (\epsilon k)^2/2) (1+O(\epsilon^{2/3})). 
\end{align}
Let us consider, for example, vector ${\mathbf z}_0$. We use the above result and and the estimate \eqref{estimate_un}
to conclude 
\begin{align*} 
256 \epsilon^{-5} \eta^{-2} \| {\mathbf z}_0\|^2&= 256 \epsilon^{-5} \eta^{-2} \sum\limits_{k \in I_N\setminus J_N} 
z_0(k)^2 + 
\epsilon \sum\limits_{k \in J_N} H_3(t_k)  \exp(- (\epsilon k)^2/2) (1+O(\epsilon^{2/3}))
\\&=
O(N^{7/2} \exp(-2N^{1/8}))+\epsilon \sum\limits_{k \in J_N} H_3(t_k)  \exp(- (\epsilon k)^2/2) (1+O(\epsilon^{2/3})).
\end{align*}
Thus, as $N\to +\infty$ we have
\begin{equation}\label{eqn_norm_z0}
256 \epsilon^{-5} \eta^{-2} \| {\mathbf z}_0\|^2 \to \int_{\r} H_3(x)^2 e^{-x^2} \d x=48 \sqrt{\pi}. 
\end{equation}
Recall that ${\boldsymbol \phi}_3={\textbf z_0}/\|{\mathbf z}_0\|$. Then \eqref{eqn_norm_z0} combined with the last formula in \eqref{eqn_wxyz0_H} show that for all $k\in J_N$ we have 
$$
\epsilon^{-1/2}{\boldsymbol \phi}_3(k)= \psi_3(\epsilon k)(1+o(1)). 
$$
In the range $k\in I_N \setminus J_N$ both the left hand side and the right-hand side are bounded by $O(N^4 \exp(-N^{1/8}))$ 
(see \eqref{estimate_un}), thus as $N\to +\infty$ both of these expressions converge to zero, uniformly in
$k\in I_N \setminus J_N$.
This concludes the proof of the convergence result \eqref{convergence_phi_psi} for $n=3$. The proof 
in the cases $n\in \{0,1,2\}$ follows exactly the same steps.

The proof for the remaining cases $n \in \{4,5,6,7\}$ will require the following simple fact, whose proof is left to the reader. 

\vspace{0.2cm}
\noindent
{\bf Fact:}
{\it 
Assume that the orthonormal vectors ${\mathbf c}$ and 
${\mathbf d}$ are obtained from two linearly independent vectors ${\mathbf a}$ and ${\mathbf b}$ by Gramm-Schmidt orthogonalization 
\begin{equation}\label{Gramm_Schmidt}
{\mathbf c}=\frac{{\mathbf a}}{\|{\mathbf a}\|}, \;\;\; \tilde {\mathbf d}={\mathbf b}-\frac{{\mathbf b} \cdot {\mathbf a}}{\|{\mathbf a}\|^2} {\mathbf a}, \;\;\; {\mathbf d}=\frac{\tilde {\mathbf d}}{\|\tilde {\mathbf d}\|}. 
\end{equation}
Then the vectors ${\mathbf c}$ and ${\mathbf d}$ would be the same if instead of ${\mathbf a}$ 
and ${\mathbf b}$ we started with $\hat {\mathbf a}=\beta_1 {\mathbf a}$ and 
$\hat {\mathbf b}=\beta_2 {\mathbf b} + \beta_3 {\mathbf a}$ for some $\beta_1>0$ and $\beta_2>0$.} 
\vspace{0.15cm}

Now, armed with this result, let us prove the convergence result \eqref{convergence_phi_psi} for $n=4$. 
According to the Definition \ref{def_phi}, the vectors ${\boldsymbol \phi}_0$ and 
${\boldsymbol \phi}_4$ are obtained by the Gramm-Schmidt orthogonalization \eqref{Gramm_Schmidt} starting from vectors 
${\textbf w}_0$ and ${\textbf w}_1$. According to the above fact, we will obtain the same result if we start with vectors  
\begin{equation}\label{def_hat_w}
\hat {\mathbf w}_0:=(2\eta)^{-1} {\mathbf w}_0, \;\;
\textnormal{ and } \; \hat {\mathbf w}_1:=(32/\eta) \epsilon^{-4} \left( {\mathbf w}_1-(2-\epsilon^2+\epsilon^4/16) {\mathbf w}_0\right).
\end{equation} 
It is clear from \eqref{eqn_wxyz0_H} that 
\begin{equation}\label{eqn_hat_w_0}
\hat w_0(k)= H_0(t_k) \exp(- (\epsilon k)^2/2) (1+O(\epsilon^{2/3}))
\end{equation}
for $k\in J_N$. Using \eqref{eqn_u0_u1} and \eqref{asymptotics_S2L_1} we compute
$$
w_1(k)=u_1(k)+N^{-1/2} S(2L+2) u_{-1}(k)=(8-\epsilon^2(4t(k)^2-2)+\epsilon^4(t(k)^2-t(k)^2-2)+O(\epsilon^6)) u_{-1}(k).
$$
Combining this result  with \eqref{wxyz0}, \eqref{eqn_u_minus1} and the following expression for the Hermite polynomial 
$$
H_4(x)=16x^4-48x^2+12,
$$
we find 
\begin{equation}
\hat  w_1(k)=(2/\eta) \times (16t_k^4-48t_k^2+12+{\mathcal E})u_{-1}(k)=
(H_4(t_k)+{\mathcal E}) \exp(- (\epsilon k)^2/2) (1+O(\epsilon^{2/3})),
\end{equation}
where the error term satisfies ${\mathcal E} = O(\epsilon^2 (1+t_k^4))$. For $k \in J_N$  we 
can estimate ${\mathcal E}=O(\epsilon)$. Thus we have proved that for 
all $k \in J_N$ 
\begin{equation}\label{eqn_hat_w_1}
\hat  w_1(k)=H_4(t_k)\exp(- (\epsilon k)^2/2) (1+O(\epsilon^{2/3})). 
\end{equation}
It is clear from \eqref{estimate_un} and \eqref{def_hat_w}
that for $k \in I_N \setminus J_N$ we have an upper bound $|w_1(k)|=O(N^2 \exp(-N^{1/8})$, thus $w_1(k)$ converges to zero 
as $N\to +\infty$ uniformly in $k \in I_N \setminus J_N$. 

Now, the vector ${\boldsymbol \phi}_4$ is given by ${\boldsymbol \phi}_4={\mathbf e}/\|{\mathbf e}\|$, where we set 
$$
{\mathbf e}:=\hat {\mathbf w}_1 - \frac{\hat {\mathbf w}_1 \cdot \hat {\mathbf w}_0}{\|\hat {\mathbf w}_0\|^2} \hat {\mathbf w}_0. 
$$
Applying \eqref{eqn_hat_w_0}, \eqref{eqn_hat_w_1}  we conclude (in the same way as we did above for 
$\|{\mathbf z}_0\|^2$) that as $N\to +\infty$
\begin{align*}
&\epsilon\|\hat {\mathbf w}_0\|^2=\epsilon \sum_{k \in I_N} \hat w_0(k)^2 \to
\int\limits_{\r} \exp(- x^2) \d x=\sqrt{\pi}, \\
&\epsilon\hat {\mathbf w}_1 \cdot \hat {\mathbf w}_0=\epsilon \sum_{k \in I_N} \hat w_0(k) \hat w_1(k) \to
\int\limits_{\r} H_4(x) \exp(- x^2) \d x=0. 
\end{align*}
Thus the quantity $\hat {\mathbf w}_1 \cdot \hat {\mathbf w}_0/\|\hat {\mathbf w}_0\|^2$ converges to zero as $N \to +\infty$, 
which shows that  $e(k) \to \hat  w_1(k)$ as $N\to +\infty$ (uniformly in $k \in I_N$). This result, combined with 
\eqref{eqn_hat_w_1} proves that the convergence result \eqref{convergence_phi_psi} holds true for $n=4$. 

The proof of \eqref{convergence_phi_psi}  for $n\in \{5,6,7\}$ follows the same steps, though the computations become 
more tedious (the use of a symbolic computation package is highly recommended for checking these expressions).  For example, when $n=7$ we use \eqref{def_wxyz} and \eqref{identity_u_n_u_n1_form2} to find 
\begin{align*}
z_0(k)&=(\epsilon^2 (2t_k^2-3)-\epsilon^4(t_k^4-5t_k^2+43/8)+\epsilon^6(t^2+1)/16+O(\epsilon^8)) \sin(2 \pi k /N) u_{-2}(k), \\
z_1(k)&=(\epsilon^2 (12 t_k^2-18)-\epsilon^4(6t_k^4-18t_k^2+57/4)+\epsilon^6(t_k^6-9t_k^4/2+11t_k^2/4+165/4)+O(\epsilon^8))
\sin(2 \pi k /N) u_{-2}(k). 
\end{align*}
We define  $\hat {\mathbf z}_1:=(512/\eta) \epsilon^{-7}({\mathbf z}_1-(6-6\epsilon^2+49\epsilon^4/16){\mathbf z}_0)$ and
 check that 
\begin{align*}
\hat z_1(k)=(4/\eta) \epsilon^{-1} (128t_k^6-1344t_k^4+3360t_k^2-1680+{\mathcal E})\sin(2 \pi k /N) u_{-2}(k) \\
= (H_7(t_k)+{\mathcal E})\exp(- (\epsilon k)^2/2) \cos(\pi k /N) (1+O(\epsilon^{2/3})),
\end{align*}
where $|{\mathcal E}|<\epsilon^2 (1+t_k^8)$. Starting with this result and 
following exactly the same argument as we have used in the case $n=4$ one can show that the convergence result 
\eqref{convergence_phi_psi}  holds true for $n=7$. The details and the proof of the remaining cases $n\in \{5,6\}$ are left to the reader.

%


\begin{appendices}

    \setcounter{proposition}{0}
    \renewcommand{\theproposition}{\Alph{section}\arabic{proposition}}
    \setcounter{theorem}{0}
    \renewcommand{\thetheorem}{\Alph{section}\arabic{theorem}}

\section{}


In this section we present the analogues of Theorem \ref{theorem_4N_plus1_n1} and Corollary \ref{prop_4N_plus1_n2} for the remaining cases $ N \in \{0, 2, 3\} \,({\textnormal{mod }}4)$. In each case
we give formulas for the even and odd vectors ${\mathbf u}_n$ and ${\mathbf v}_n$ (which were described in Theorem \ref{thm_support}) and compute their corresponding discrete Fourier transforms. 
The proof of these results is based on Lemma \ref{main_lemma} and it is very similar to the proof of Theorem \ref{theorem_4N_plus1_n1}, we omit all the details.  We construct a basis for each eigenspace $W$, $X$, $Y$ and $Z$, from which one can 
build the Hermite-type eigenvectors ${\boldsymbol \phi}_n$  in exactly the same way as in Definition \ref{def_phi}. 

We recall that the interval $I_N$ and the discrete Fourier transform  ${\mathcal F}$ are defined by \eqref{def_Fourier_transform} and 
the sequence $S(k)$ is defined by \eqref{def_sk}.

\subsection{The case when $N=4L$}

We define $2L+1$ vectors $\{{\mathbf u}_n\}_{-L \le n \le L}$ and $2L-1$ vectors $\{{\mathbf v}_n\}_{-L+1 \le n \le L-1}$ as follows: 
for $-2L+1\le k \le 2L$ we set $u_{-L}(k):=2L$ and 
\begin{align*}
u_n(k)&:=\cos(\pi k/N) S(3L+n-1+k)S(3L+n-1-k), \;\;\; -L+1\le n \le L, \\
v_n(k)&:=\sin(\pi k/N) S(3L+n-1+k)S(3L+n-1-k), \;\;\; -L+1 \le n \le L-1. 
\end{align*}
These $N$ vectors are linearly independent and they satisfy 
\begin{align*}
{\mathcal F} {\mathbf u}_n&=N^{-1/2}\frac{S(2L+2n)}{2\sin(\pi(3L+n)/N)} {\mathbf u}_{-n}, \;\;\; -L \le n < L-1, \\
{\mathcal F} {\mathbf v}_n&=-\i N^{-1/2} \frac{S(2L+2n)}{2\sin(\pi(3L-n)/N)} {\mathbf v}_{-n}, \;\;\; -L+1 \le n \le L-1, 
\end{align*}
and ${\mathcal F}{\mathbf u}_L=2\sqrt{N} {\mathbf u}_{-L}$. The vectors 
$\{{\mathbf u}_n\}_{-L+1 \le n \le L+1}$ and $\{{\mathbf v}_n\}_{-L+1 \le n \le L-1}$ 
are precisely the vectors described in Theorem \ref{thm_support}(ii). 
An equivalent expression for the vectors ${\mathbf u}_n$ is given by
\begin{equation*}
u_n(k)=N2^{L+n} \cos(\pi k/N)^2 \prod\limits_{j=L-n+1}^{2L-1} \left( \cos(2\pi k/N)-\cos(2\pi j/N)\right). 
\end{equation*}

Define $N$ vectors $\{{\mathbf w}_n\}_{0\le n \le L}$, $\{{\mathbf x}_n\}_{0\le n \le L-1}$, 
$\{{\mathbf y}_n\}_{0\le n \le L-1}$, $\{{\mathbf z}_n\}_{0\le n \le L-2}$ as follows:
\begin{align*}
&{\mathbf w}_n={\mathbf u}_n+N^{-1/2}\frac{S(2L+2n)}{2\sin(\pi(3L+n)/N)} {\mathbf u}_{-n}\;, \;\; 0 \le n \le L-1,\\
&{\mathbf w}_L={\mathbf u}_L+2N^{1/2} {\mathbf u}_{-L}, \\
&{\mathbf x}_n={\mathbf v}_n+N^{-1/2}\frac{S(2L+2n)}{2\sin(\pi(3L-n)/N)} {\mathbf v}_{-n}\;, \;\; 0 \le n \le L-1,\\
&{\mathbf y}_n=-{\mathbf u}_{n+1}+N^{-1/2}\frac{S(2L+2n+2)}{2\sin(\pi(3L+n+1)/N)} {\mathbf u}_{-n-1}\;, \;\; 0 \le n \le L-2, \\
&{\mathbf y}_{L-1}=-{\mathbf u}_L+2N^{1/2} {\mathbf u}_{-L},\\
&{\mathbf z}_n=-{\mathbf v}_{n+1}+N^{-1/2}\frac{S(2L+2n+2)}{2\sin(\pi(3L-n-1)/N)} {\mathbf v}_{-n-1}\;, \;\; 0 \le n \le L-2.
\end{align*}
The vectors $\{{\mathbf w}_n\}_{0\le n \le L}$ (respectively, $\{{\mathbf x}_n\}_{0\le n \le L-1}$, 
$\{{\mathbf y}_n\}_{0\le n \le L-1}$ and $\{{\mathbf z}_n\}_{0\le n \le L-2}$)
form a basis for the eigenspace $W$ (respectively, $X$, $Y$ and $Z$).

\subsection{The case when $N=4L+2$}

We define $2L+2$ vectors $\{{\mathbf u}_n\}_{-L \le n \le L+1}$ and $2L$ vectors $\{{\mathbf v}_n\}_{-L+1 \le n \le L}$ as follows: 
for $-2L\le k \le 2L+1$ we set $u_{-L}(k):=2L+1$ and 
\begin{align*}
u_n(k)&:=\cos(\pi k/N) S(3L+n+k)S(3L+n-k), \;\;\; -L+1\le n \le L+1, \\
v_n(k)&:=\sin(\pi k/N) S(3L+n+k)S(3L+n-k), \;\;\; -L+1 \le n \le L.
\end{align*}
These $N$ vectors are linearly independent and they satisfy 
\begin{align*}
{\mathcal F} {\mathbf u}_n&=N^{-1/2}\frac{S(2L+2n)}{2\sin(\pi(3L+n+1)/N)} {\mathbf u}_{1-n}, \;\;\; -L \le n \le L, \\ \\
{\mathcal F} {\mathbf v}_n&=-\i N^{-1/2}\frac{S(2L+2n)}{2\sin(\pi(3L-n+2)/N)} {\mathbf v}_{1-n}, \;\;\; -L+1 \le n \le L,
\end{align*}
and ${\mathcal F}{\mathbf u}_{L+1}=2N^{1/2} {\mathbf u}_{-L}$.  The vectors 
$\{{\mathbf u}_n\}_{-L+1 \le n \le L}$ and $\{{\mathbf v}_n\}_{-L+1 \le n \le L}$ 
are precisely the vectors described in Theorem \ref{thm_support}(ii). 
An equivalent expression for the vectors ${\mathbf u}_n$ is given by
\begin{equation*}
u_n(k)=N2^{L+n} \cos(\pi k/N)^2 \prod\limits_{j=L-n+2}^{2L} \left( \cos(2\pi k/N)-\cos(2\pi j/N)\right).
\end{equation*}

Define $N$ vectors $\{{\mathbf w}_n\}_{0\le n \le L}$, $\{{\mathbf x}_n\}_{0\le n \le L-1}$, 
$\{{\mathbf y}_n\}_{0\le n \le L}$, $\{{\mathbf z}_n\}_{0\le n \le L-1}$ as follows:
\begin{align*}
&{\mathbf w}_n={\mathbf u}_{n+1}+N^{-1/2}\frac{S(2L+2n+2)}{2\sin(\pi(3L+n+2)/N)} {\mathbf u}_{-n}\;, \;\; 0 \le n \le L-1,\\
&{\mathbf w}_L={\mathbf u}_{L+1}+2N^{-1/2} {\mathbf u}_{-L}, \\
&{\mathbf x}_n={\mathbf v}_{n+1}+N^{-1/2}\frac{S(2L+2n+2)}{2\sin(\pi(3L-n+1)/N)} {\mathbf v}_{-n}\;, \;\; 0 \le n \le L-1,\\
&{\mathbf y}_n=-{\mathbf u}_{n+1}+N^{-1/2}\frac{S(2L+2n+2)}{2\sin(\pi(3L+n+2)/N)} {\mathbf u}_{-n}\;, \;\; 0 \le n \le L-1, \\
&{\mathbf y}_L=-{\mathbf u}_{L+1}+2N^{-1/2} {\mathbf u}_{-L}, \\
&{\mathbf z}_n=-{\mathbf v}_{n+1}+N^{-1/2}\frac{S(2L+2n+2)}{2\sin(\pi(3L-n+1)/N)} {\mathbf v}_{-n}\;, \;\; 0 \le n \le L-1.
\end{align*}
The vectors $\{{\mathbf w}_n\}_{0\le n \le L}$ (respectively, $\{{\mathbf x}_n\}_{0\le n \le L-1}$, 
$\{{\mathbf y}_n\}_{0\le n \le L}$ and $\{{\mathbf z}_n\}_{0\le n \le L-1}$)
form a basis for the eigenspace $W$ (respectively, $X$, $Y$ and $Z$).

\subsection{The case when $N=4L+3$}

We define $2L+2$ vectors $\{{\mathbf u}_n\}_{-L \le n \le L+1}$ and $2L+1$ vectors $\{{\mathbf v}_n\}_{-L \le n \le L}$ as follows: 
for $-2L-1 \le k \le 2L+1$ we set
\begin{align*}
u_n(k)&:= S(3L+n+1+k)S(3L+n+1-k), \;\;\; -L \le n \le L+1, \\
v_n(k)&:=\sin(2\pi k/N) S(3L+n+1+k)S(3L+n+1-k), \;\;\; -L \le n \le L. 
\end{align*}
These $N$ vectors are linearly independent and they satisfy 
\begin{align*}
{\mathcal F} {\mathbf u}_n&=N^{-1/2}S(2L+2n) {\mathbf u}_{1-n}, \;\;\; -L \le n \le L+1 \\
{\mathcal F} {\mathbf v}_n&=-\i N^{-1/2}S(2L+2n+1) {\mathbf v}_{-n}, \;\;\; -L \le n \le L. 
\end{align*}
The vectors 
$\mathbf u_n$ and $\mathbf v_n$ 
are precisely the vectors described in Theorem \ref{thm_support}(i). 
An equivalent expression for the vectors ${\mathbf u}_n$ is given by
\begin{equation}
u_n(k)=N2^{L+n} \prod\limits_{j=L-n+2}^{2L} \left( \cos(2\pi k/N)-\cos(2\pi j/N)\right).
\end{equation}

Define $N$ vectors $\{{\mathbf w}_n\}_{0\le n \le L}$, $\{{\mathbf x}_n\}_{0\le n \le L}$, 
$\{{\mathbf y}_n\}_{0\le n \le L}$, $\{{\mathbf z}_n\}_{0\le n \le L-1}$ as follows:
\begin{align*}
&{\mathbf w}_n={\mathbf u}_{n+1}+N^{-1/2}S(2L+2n+2) {\mathbf u}_{-n}\;, \;\; 0 \le n \le L,\\
&{\mathbf x}_n={\mathbf v}_n+N^{-1/2}S(2L+2n+1) {\mathbf v}_{-n}\;, \;\; 0 \le n \le L,\\
&{\mathbf y}_n=-{\mathbf u}_{n+1}+N^{-1/2}S(2L+2n+2) {\mathbf u}_{-n}\;, \;\; 0 \le n \le L, \\\
&{\mathbf z}_n=-{\mathbf v}_{n+1}+N^{-1/2}S(2L+2n+3) {\mathbf v}_{-n-1}\;, \;\; 0\le n \le L-1.
\end{align*}
The vectors $\{{\mathbf w}_n\}_{0\le n \le L}$ (respectively, $\{{\mathbf x}_n\}_{0\le n \le L}$, 
$\{{\mathbf y}_n\}_{0\le n \le L}$ and $\{{\mathbf z}_n\}_{0\le n \le L-1}$)
form a basis for the eigenspace $W$ (respectively, $X$, $Y$ and $Z$). 

\end{appendices}


\end{document}